\documentclass[11pt,letterpaper]{article}
\usepackage[T1]{fontenc}
\usepackage[utf8]{inputenc}
\usepackage{lmodern}
\usepackage[margin=0.91in,top=0.84in,bottom=0.9in]{geometry}
\usepackage{amsmath,amssymb,amsthm,mathtools}
\usepackage{enumitem}
\usepackage{microtype}
\usepackage[bookmarksnumbered,hidelinks]{hyperref}
\hypersetup{pdftitle={Dedekind Eta Powers and Interlacing in Lambert Series Quotients},pdfauthor={K. Srinivasa Raghava}}
\numberwithin{equation}{section}
\newtheorem{theorem}{Theorem}[section]
\newtheorem{proposition}[theorem]{Proposition}
\newtheorem{lemma}[theorem]{Lemma}
\newtheorem{corollary}[theorem]{Corollary}
\newtheorem{conjecture}[theorem]{Conjecture}
\theoremstyle{definition}
\newtheorem{remark}[theorem]{Remark}
\DeclareMathOperator{\sech}{sech}
\setlist{nosep}
\allowdisplaybreaks[1]
\DeclareMathOperator{\Res}{Res}
\newcommand{\LL}{\mathcal{L}}
\newcommand{\cN}{\mathcal{N}}
\newcommand{\cD}{\mathcal{D}}
\newcommand{\re}{\operatorname{Re}}
\newcommand{\im}{\operatorname{Im}}
\newcommand{\CC}{\mathbb{C}}
\newcommand{\RR}{\mathbb{R}}
\newcommand{\ZZ}{\mathbb{Z}}
\setlist[enumerate]{label=\textup{(\roman*)},leftmargin=2.3em}
\newcommand{\doi}[1]{\href{https://doi.org/#1}{\nolinkurl{https://doi.org/#1}}}
\newcommand{\arxiv}[1]{\href{https://arxiv.org/abs/#1}{\nolinkurl{arXiv:#1}}}

\title{Dedekind Eta Powers and Interlacing in Lambert Series Quotients}
\author{K. Srinivasa Raghava\\[2pt]
\small Pie Mathematics Association, Chennai, India\\
\small\href{mailto:srinivasaraghavak@gmail.com}{srinivasaraghavak@gmail.com}}
\date{}
\begin{document}
\maketitle
\begin{abstract}
For each \(0<q<1\) we prove that the zeros and poles of a quotient of Lambert series formed from the cube of Euler's product are simple, negative and strictly interlacing. The zeros of the sampled series decrease strictly as its shift increases through \((0,1]\). The quotient has a partial-fraction expansion with positive weights, and every average over the \(d\)-th roots of unity with \(d\) odd exceeds the value of the quotient at the origin. For even \(d\), every pole of the average in the parameter \(q\in(0,1)\) is simple and has positive residue. We also characterize the cube as the unique positive integer power of Euler's product whose sampling kernel is a P\'olya frequency function. For a general positive integer power \(k\) and integer dilation \(a\ge2\), the quotient interpolates \(z^{1-1/a}\) at geometric nodes selected by the nonzero coefficients of the Euler product power. We determine its radial limit away from two real rays and the asymptotic locations of the poles of the averages of even order near \(q=1\). For the cube, convergence to the fractional power is exponentially fast on compact subsets of \(\CC\setminus(-\infty,0]\). The positive discrete measures in the partial-fraction expansion converge vaguely on \((0,\infty)\) to \(\pi^{-1}\sin(\pi/a)u^{-1/a}\,du\).
\end{abstract}


\section{Introduction}
\subsection{The quotient and the cubic interlacing theorem}
Let \(\Phi(x)=\prod_{n\ge1}(1-x^n)\), and write \(\Phi(x)^k=\sum_{n\ge0}f_nx^n\).
Fix integers \(a\ge2\), \(k\ge1\), and let \(0<q<1\). We sample \(x\Phi(x)^k\)
on two geometric progressions of common ratio \(q^a\). Put
\begin{equation}\label{eq:J}
\begin{split}
 N(z,q)&=\sum_{m\ge0}z^m q^{am+1}\Phi(q^{am+1})^k,\qquad
 D(z,q)=\sum_{m\ge0}z^m q^{am+a}\Phi(q^{am+a})^k,\\[-2pt]
 J(z,q)&=N(z,q)/D(z,q).
\end{split}
\end{equation}
We suppress \(q\) when it is fixed. The value at the origin is
\begin{equation}\label{eq:T}
 T(q)=J(0,q)=q^{1-a}\frac{\Phi(q)^k}{\Phi(q^a)^k}.
\end{equation}
For \(q=e^{2\pi i\tau}\) and \(\eta(\tau)=q^{1/24}\Phi(q)\), this equals
\(q^{(a-1)(k-24)/24}(\eta(\tau)/\eta(a\tau))^k\) \cite{Ono}.
For \(k=24\) this gives \(T(q)=\Delta(\tau)/\Delta(a\tau)\), where
\(\Delta(\tau)=q\Phi(q)^{24}=\sum_{n\ge1}\tau(n)q^n\) \cite{Ramanujan}.
For each integer \(d\ge1\) let
\begin{equation}\label{eq:W}
 W_d(q)=\frac1d\sum_{\omega^d=1}J(\omega,q)
\end{equation}
be the average over the \(d\)-th roots of unity, defined whenever every summand is finite. If \(J(z)=\sum_{m\ge0}c_mz^m\) converges on the unit
circle, the average retains \(c_0,c_d,c_{2d},\ldots\), with \(c_0=T\).

\paragraph{Interlacing for the cube of Euler's product.}
For \(k=3\), fix any \(q\in(0,1)\). The zeros \(-v_i\) and poles \(-u_i\) of the
meromorphic function \(J(\,\cdot\,,q)\) are all simple and exhaust two sequences
on the negative real axis, ordered by
\[
 0<v_1<u_1<v_2<u_2<\cdots .
\]
There are positive weights \(w_i\) with \(\sum_iw_i/u_i<\infty\) such that
\[
 J(z,q)=T(q)+\sum_{i\ge1} w_i\frac{z}{z+u_i},
 \qquad W_d(q)=T(q)+\sum_{i\ge1}\frac{w_i}{1-(-u_i)^d}.
\]
The second formula holds whenever \(W_d\) is finite. For odd \(d\), \(W_d\) is real analytic and strictly greater than \(T\) throughout \(0<q<1\).
For even \(d\), a pole occurs precisely when some \(u_i(q)=1\). Corollary~\ref{cor:polemotion} gives \(u_i'(q)\le-a u_i(q)/q\), so every such pole is simple with positive residue. Before the first
such \(q\), \(W_{2m}\) increases to \(T\) and \(W_{2m+1}\) decreases to \(T\) as \(m\to\infty\).
To prove interlacing, write \(G(v)=e^{-v}\Phi(e^{-v})^3\) for \(v>0\).
The zeros of \(\sum_{m\ge0}z^mG(t(m+x))\), continued meromorphically in \(z\),
decrease strictly with \(x\in(0,1]\) at fixed \(t=-a\log q\)
(Theorem~\ref{thm:zeros}). Theorems~\ref{thm:pick} and~\ref{thm:sections} give the two expansions above.

\subsection{Interpolation and limits as the parameter varies}
The Lambert series expansions \eqref{eq:lambertQ} give meromorphic continuations of \(N\) and \(D\) in \(z\). At \(z=q^{-an}\) with \(f_{n-1}\ne0\), the two simple poles have residues
in the ratio \(q^{-(a-1)n}\). Thus
\begin{equation}\label{eq:interp}
 J(q^{-an},q)=(q^{-an})^{1-1/a}\qquad (n\ge1,\ f_{n-1}\ne0).
\end{equation}
We write \(\beta=1-1/a\). In the cubic case \(J(z,q)/z\) is the Stieltjes
transform of \(T\delta_0+\sum_i w_i\delta_{u_i}\); the corresponding
measure for \(z^{\beta-1}\) has density \(\pi^{-1}\sin(\pi/a)u^{-1/a}\).
Theorem~\ref{thm:limit} identifies the vague limit of \(\sum_iw_i\delta_{u_i}\). In particular, along
values of \(q\) for which \(D(-1,q)\ne0\),
\[
 \lim_{q\to1^-}\lim_{d\to\infty}W_d(q)=
 \frac{a\sin(\pi/a)}{\pi(a-1)}.
\]
The limit is \(2/\pi\) for \(a=2\).

For \(k\ge2\), \(W_d-T\) vanishes at \(q=0\) to order \(2a(d-1)+2\), with leading coefficient \(k(1-k)^{d-1}\). For \(k=1\) the order depends on the parity of \(d\). Section~\ref{sec:contact} proves both formulas and gives the divisibility of the intervening coefficients.

Put \(q=e^{-t/a}\). For every \(k\ge1\) and
\(z\notin(-\infty,0]\cup(1,\infty)\), Theorem~\ref{thm:radial} gives
\begin{equation}\label{eq:intro-limit}
 J(z,e^{-t/a})=z^\beta\bigl(1+O(e^{-c/\sqrt t})\bigr)\quad(t\downarrow0)
\end{equation}
for every \(c<c(z)\), with the positive rate \(c(z)\) in \eqref{eq:cz}.
Theorem~\ref{thm:limit}(i) extends this estimate to the whole slit plane for \(k=3\)
with \(c=c(z)\), uniformly on compact sets. Averaging over the \(d\)-th roots of unity for odd \(d\) gives
\begin{equation}\label{eq:intro-odd}
 \lim_{q\to1^-}W_d(q)=\frac{\sin(\pi/a)}{d\sin\bigl(\pi(a-1)/(ad)\bigr)}
 \qquad(d\ \text{odd}).
\end{equation}
For even \(d\), the average includes \(J(-1,q)\). At this point the two leading Fourier terms have equal magnitude. The zeros of their sum determine the asymptotic locations of infinitely many simple poles \(q_j=e^{-t_j/a}\) approaching \(1\), with \(1-q_j\sim\pi k/(3aj^2)\) (Theorem~\ref{thm:poles}). For \(k=3\), Theorem~\ref{thm:tj} determines the leading oscillatory correction to \(t_j\).

\subsection{Relation to earlier preprints}
In the October 2025 preprint
\href{https://www.researchgate.net/publication/397097749}{\emph{From \(\eta\)-replicates to cyclotomic symmetry}},
the author
studied the average \eqref{eq:W} with \(a=d=2\).
Theorem 3 there asserted a finite radial limit at every root of unity.
At \(q=1\) that assertion is false for every positive integer \(k\):
Theorem~\ref{thm:poles} below gives simple real poles accumulating at \(1\) along the
radius in question. The author's January 2026 preprint
\emph{Evenized dilation quotients from Euler products}
(\href{https://www.researchgate.net/publication/400079147}{ResearchGate record 400079147})
announced the coefficient of \(q^6\) in \(W_2-T\), together with modular and complex-multiplication approximations for the same average.

The first version of this paper, \emph{Root-of-Unity Sections of Dilated Euler-Product Quotients}
(\href{https://arxiv.org/abs/2609.22346v1}{arXiv:2609.22346v1}, September 2026),
treats general \(a\) and \(d\). Its Section~2 gives the orders of vanishing and coefficient divisibility at \(q=0\). Sections~4 and~5 give the limits for odd \(d\), the asymptotic locations and residues of the poles for even \(d\), and the oscillatory correction for \(k=3\). Theorem~3.1 there locates the nearest pole in \(z\) for \(k\ge2\) and small \(q\). The conclusion asks whether the averages of odd order have poles between the two endpoint neighborhoods. Here Theorems~\ref{thm:zeros}, \ref{thm:pick} and~\ref{thm:sections} settle that question for \(k=3\) throughout \(0<q<1\). Theorem~\ref{thm:limit} identifies the limit of the positive measures in the partial-fraction expansion.

\subsection{Background and notation}
Poisson summation expresses the sampled series in terms of the Laplace transform of \(G(v)=e^{-v}\Phi(e^{-v})^k\). The $\eta$ transformation gives the decay at zero needed for this step \cite{Apostol}. For \(k=1,3\) the transforms reduce to Glasser's formulas \cite{Glasser}. In the cubic case the reciprocal transform is an entire product with negative zeros, so Schoenberg's theorem gives total positivity \cite{Sch51,BGKP}. We use the Aissen--Schoenberg--Whitney--Edrei theorem and variation diminution to study the zeros after sampling \cite{ASW,Edrei,Sch30,Karlin}. Comparison with finite cardinal \(L\)-splines gives the dependence of the pole radii on \(q\) \cite{MicchelliSharma}. Proposition~\ref{prop:onlycube} explains why total positivity singles out the cube among all positive integer powers.

On the unit circle the sampled functions are values of a Zak transform \cite[Chapter~8]{GroBook}. For a continuous P\'olya frequency function of infinite type without a Gaussian factor this transform has a unique zero in a fundamental domain \cite{Kloos,KS}. Remark~\ref{rem:zak} relates this result to the zeros at other radii. The interlacing theorem yields their multiplicities, order and dependence on the sampling shift.

Throughout, \(a\ge2\) and \(k\ge1\) are integers unless stated otherwise.
We have \(f_0=1,\ f_1=-k,\ f_2=k(k-3)/2\). Except in Section 3 we take
\(0<q<1\) and write \(Q=q^a=e^{-t}\). Powers and logarithms are principal
values, and square roots are taken in the closed right half plane.

\section{Lambert series and interpolation}\label{sec:interp}

In this section $0<q<1$ and $Q=q^a$.

\begin{proposition}\label{prop:lambert}
\begin{enumerate}
\item The series \eqref{eq:J} converge for $|z|<Q^{-1}$, and there
\begin{equation}\label{eq:lambertQ}
N(z)=\sum_{n\ge1}\frac{f_{n-1}\,q^{n}}{1-zQ^{n}},\qquad
D(z)=\sum_{n\ge1}\frac{f_{n-1}\,Q^{n}}{1-zQ^{n}} .
\end{equation}
The series in \eqref{eq:lambertQ} converge locally uniformly on $\CC\setminus\{Q^{-n}:n\ge1\}$. Hence $N$ and $D$ are meromorphic in $\CC$ with at most simple poles at the points $Q^{-n}$, the function $D$ is not identically zero, and $J=N/D$ is meromorphic in $\CC$.
\item If $f_{n-1}\neq 0$, then $J$ is finite at $z=Q^{-n}$ and $J(Q^{-n})=(q/Q)^n$. This is \eqref{eq:interp}.
\end{enumerate}
\end{proposition}

\begin{proof}
Since $q^{(am+1)(n+1)}\le q^{n+1}Q^{m}$ and $\sum_n|f_n|q^{n+1}<\infty$, the double series
\[
\sum_{m,n\ge0}z^m f_n\,q^{(am+1)(n+1)}
\]
converges absolutely for $|z|<Q^{-1}$. Summing over $n$ first gives $N(z)$, and summing over $m$ first gives the first series in \eqref{eq:lambertQ}. The same computation with $Q^{m+1}$ in place of $q^{am+1}$ gives the formula for $D$. On a compact set which contains none of the points $Q^{-n}$ the factors $1/(1-zQ^n)$ are bounded uniformly in $n$, and this gives the locally uniform convergence. At $z=Q^{-n}$ the residues of $N$ and $D$ are $-f_{n-1}(q/Q)^n$ and $-f_{n-1}$. For $n=1$ this shows that $D$ is not identically zero. If $f_{n-1}\ne0$, then $N$ and $D$ both have a simple pole at $Q^{-n}$, and the value of $J$ there is the quotient of the residues.
\end{proof}

We call the points $Q^{-n}$ with $f_{n-1}\ne0$ the \emph{nodes} of $J$.

\begin{remark}\label{rem:bary}
With $z_n=Q^{-n}$, formula \eqref{eq:lambertQ} reads
\[
J(z)=\frac{\sum_{n\ge1} f_{n-1}\,(q/Q)^n/(z_n-z)}{\sum_{n\ge1} f_{n-1}/(z_n-z)} .
\]
The finite truncations are barycentric rational interpolants \cite{BT}.
The nodes in the infinite formula are indexed by the nonzero coefficients of
$\Phi^k$. Euler's and Jacobi's identities describe the nodes for $k=1,3$
\cite{Andrews}. Serre classified the lacunary cases for even $k$ \cite{Serre}.
For $k=24$, the question whether every $Q^{-n}$ is a node is Lehmer's question
about the nonvanishing of $f_{n-1}=\tau(n)$ \cite{Lehmer,HNW}.

If consecutive nodes $z_n<z_{n'}$ have weights of the same sign, $D$ tends to
infinities of opposite signs at the two ends of $(z_n,z_{n'})$ and has a zero
there. This happens at $n=2$, $n'=3$ for $k=1,2$, and at $n=6$, $n'=7$
for $k=24$. For $k=3$ the weights alternate, and Section~\ref{sec:tp}
shows that all zeros of $D$ are negative.

For $k=2$ the quotient itself has a positive pole for all sufficiently small
$q$. Indeed, $f_0=1$, $f_1=-2$, $f_2=-1$. Put $z=yQ^{-3}$.
By \eqref{eq:lambertQ}, uniformly on $|y-1/2|\le1/4$ as $q\to0^+$,
\[
 Q^{-3}D(yQ^{-3})\longrightarrow \frac1y-\frac1{1-y},
 \qquad q^{-3}N(yQ^{-3})\longrightarrow -\frac1{1-y}.
\]
Here the terms with $n\ge4$ in the scaled series are $O(Q)$ and $O(q)$,
respectively. The first limit has a simple zero at $1/2$, and the second
is nonzero throughout the disc. Rouché's theorem and complex conjugation
therefore give a simple real pole $z\sim\tfrac12 Q^{-3}$.
\end{remark}

If $J(\cdot,q)$ is analytic on the closed unit disc and $J(z)=\sum_{m\ge0}c_mz^m$ there, then \eqref{eq:W} gives
\begin{equation}\label{eq:Wc}
W_d=\sum_{j\ge 0}c_{jd},\qquad W_d-T=c_d+c_{2d}+c_{3d}+\cdots .
\end{equation}
The next statement gives a range of $q$ in which this applies, for every $k$.

\begin{proposition}\label{prop:EK}
Let $\varrho=Q\prod_{n\ge1}(1+Q^n)^k$. Then $D(z)\ne 0$ for $|z|\le 1/\varrho$. If $\varrho<1$, then for each $r$ with $\varrho<r<1$ there is a constant $C$ such that $|W_d-T|\le C\,r^{d}$ for all $d\ge1$. In particular $W_d\to T$ as $d\to\infty$.
\end{proposition}

\begin{proof}
Write $D(z)=\sum_{m\ge0}b_mz^m$ with $b_m=Q^{m+1}\Phi(Q^{m+1})^k>0$. Then
\[
\frac{b_{m+1}}{b_m}=Q\prod_{n\ge1}\Bigl(\frac{1-Q^{n(m+2)}}{1-Q^{n(m+1)}}\Bigr)^{k}.
\]
For fixed $n$ the factor is $(1-uQ^n)/(1-u)$ with $u=Q^{n(m+1)}$. It increases with $u$, and $u$ decreases with $m$. So the ratios $b_{m+1}/b_m$ decrease strictly, from the value $\varrho$ at $m=0$ to the limit $Q$. Hence $\varrho b_{m-1}-b_m\ge0$ for $m\ge1$, with strict inequality for $m\ge2$, and
\[
(1-\varrho z)D(z)=b_0-\sum_{m\ge1}(\varrho b_{m-1}-b_m)z^m .
\]
For $|z|\le 1/\varrho$ the absolute value of the sum is at most $\sum_{m\ge1}(\varrho b_{m-1}-b_m)\varrho^{-m}=b_0$; the series telescopes, and $b_m\varrho^{-m}\to0$ because $b_m\le Q^{m+1}$ and $Q<\varrho$. If the left side vanishes, then the sum equals $b_0$. This forces $|z|=1/\varrho$, $z^2>0$ and $z^3>0$, that is $z=1/\varrho$. But $D(1/\varrho)>0$. So $D$ has no zero in the disc $|z|\le1/\varrho$. If $\varrho<1$, then $J$ is analytic in a disc of radius larger than $1/r$. Cauchy's inequality gives $|c_m|\le C' r^{m}$, and \eqref{eq:Wc} gives the bound.
\end{proof}

For $k=24$ and $\tau=i/\sqrt a$, the transformation law
$\Delta(-1/\tau)=\tau^{12}\Delta(\tau)$ gives $T=a^6$.
Since $Q=e^{-2\pi\sqrt a}<1/100$ and
$\log\varrho\le\log Q+24Q/(1-Q)<0$, Proposition~\ref{prop:EK} gives
$W_d\to a^6$ as $d\to\infty$.

\section{Root-of-unity averages near \texorpdfstring{$q=0$}{q=0}}\label{sec:contact}

In this section $q$ is a formal variable, $Q=q^a$, and $R$ is a commutative ring with $1$. Let $F=\sum_{n\ge0}f_nq^n\in R[[q]]$ with $f_0=1$. In our case $F=\Phi^k$, but only the first coefficients of $F$ will matter. Guided by \eqref{eq:lambertQ} we define
\[
N=\sum_{n\ge0}\frac{f_n\,q^{n+1}}{1-zQ^{n+1}},\qquad D=\sum_{n\ge0}\frac{f_n\,Q^{n+1}}{1-zQ^{n+1}}
\]
as power series in $z$ with coefficients in $R[[q]]$. The coefficient of $z^m$ in $N$ is $q^{am+1}F(q^{am+1})$, as in \eqref{eq:J}. The constant term of $D/Q$ is $F(Q)$, a unit of $R[[q]]$. So
\[
J=\frac ND=\sum_{m\ge0}c_mz^m,\qquad c_m\in q^{1-a}R[[q]],\qquad c_0=T=q^{1-a}\frac{F(q)}{F(Q)} .
\]
By Lemma \ref{lem:reduced} below, the order of $c_m$ in $q$ tends to infinity with $m$. So the sum $W_d=\sum_{j\ge0}c_{jd}$ is a well defined Laurent series in $q$. We write $O(q^e)$ for an element of $q^eR[[q]]$; the exponent $e$ may be negative.

For $F=\Phi^k$ this formal series is the Laurent expansion at $q=0$ of the average \eqref{eq:W}. Indeed, the definitions \eqref{eq:J} to \eqref{eq:W} make sense for complex $q$ with $0<|q|<1$. For small $|q|$ we have $D(z)=Q(1+O(Q))$ and $N(z)=O(q)$ uniformly on $|z|\le2$. So $J(\cdot,q)$ is analytic on this disc, its Taylor coefficients $c_m(q)$ are analytic in a punctured disc $0<|q|<\epsilon$ and satisfy $|c_m(q)|\le C|q|^{1-a}2^{-m}$ there, and the series $W_d(q)=\sum_jc_{jd}(q)$ from \eqref{eq:Wc} converges uniformly after multiplication by $q^{a-1}$. Hence the Laurent series of $W_d$ at $q=0$ is the sum of the Laurent series of the $c_{jd}$.

\begin{lemma}\label{lem:reduced}
Put $Y=zQ^2$. In the ring $R((q))[[Y]]$ we have $J=q^{1-a}\,\cN/\cD$, where
\begin{align*}
\cD&=1-(1+f_1)Y+f_1Q+(1-Y)\sum_{n\ge2}\frac{f_n\,Q^{n-1}(Q-Y)}{1-Q^{n-1}Y},\\
\cN&=1-Y+f_1q-f_1q^{1-a}\,Y+(1-Y)\sum_{n\ge2}\frac{f_n\,q^{n-a}(Q-Y)}{1-Q^{n-1}Y} .
\end{align*}
The following coefficient estimates hold.
\begin{enumerate}
\item $\cD=1-(1+f_1)Y+\sum_{i\ge1}D_i(Y)\,Q^i$, where $D_i\in R[Y]$ has degree at most $i+1$, and the coefficient of $Y^{i+1}$ in $D_i$ is $f_2$.
\item $\cN=\sum_{j\ge0}\nu_jY^j$ with $\nu_0=F(q)$, $\nu_1=-f_1q^{1-a}+O(q^{2-a})$, and $\nu_j=f_2\,q^{a(j-3)+2}+O\bigl(q^{a(j-3)+3}\bigr)$ for $j\ge2$.
\item $\cN/\cD\in q^{1-a}R[[q]][[Y]]$, and $c_m$ is $q^{1-a}Q^{2m}$ times the coefficient of $Y^m$ in $\cN/\cD$. In particular $c_m=O\bigl(q^{2a(m-1)+2}\bigr)$.
\end{enumerate}
\end{lemma}

\begin{proof}
Since $Q$ is a unit of $R((q))$, the substitution $z=Y/Q^2$ identifies $R((q))[[z]]$ with $R((q))[[Y]]$. It turns $1/(1-zQ^{n+1})$ into $1/(1-Q^{n-1}Y)$, so that
\[
\frac DQ=\frac{1}{1-Y/Q}+\frac{f_1Q}{1-Y}+\sum_{n\ge2}\frac{f_nQ^n}{1-Q^{n-1}Y},\qquad
\frac Nq=\frac{1}{1-Y/Q}+\frac{f_1q}{1-Y}+\sum_{n\ge2}\frac{f_nq^n}{1-Q^{n-1}Y} .
\]
Multiplying by the unit $(1-Y/Q)(1-Y)$ gives $\cD$ and $\cN$, and $J=N/D=(q/Q)\,\cN/\cD$.

For $n\ge2$ we have
\[
(1-Y)(Q-Y)\sum_{l\ge0}Q^{(n-1)l}Y^l=\sum_{l\ge0}Q^{(n-1)l}\bigl[Q\,Y^l-(1+Q)\,Y^{l+1}+Y^{l+2}\bigr].
\]
In $\cD$ this is multiplied by $f_nQ^{n-1}$. So a power $Y^h$ with $h\ge2$ occurs with the exponent $(n-1)(h-1)$ of $Q$ through the last term in the bracket, and with exponents at least $(n-1)h$ through the other two. All these exponents are at least $h-1$, and the value $h-1$ is attained only for $n=2$ in the last term. This proves (i). In $\cN$ the same expression is multiplied by $f_nq^{n-a}$. So for $j\ge2$ the power $Y^j$ occurs with the exponent $n-a+a(n-1)(j-2)$ of $q$ through the last term, and with larger exponents otherwise. The smallest exponent is $a(j-3)+2$; it belongs to $n=2$ and has the coefficient $f_2$. The next one is at least $a(j-3)+3$. The coefficients $\nu_0$ and $\nu_1$ are read off in the same way; note that $a\ge2$. This proves (ii).

The constant term of $\cD$ with respect to $Y$ is $F(Q)$, so $1/\cD\in R[[q]][[Y]]$, and $\cN\in q^{1-a}R[[q]][[Y]]$ by (ii). Finally $z^m=Q^{-2m}Y^m$. This proves (iii).
\end{proof}

\begin{theorem}\label{thm:contact}
For $m\ge1$,
\[
c_m=-f_1(1+f_1)^{m-1}\,q^{2a(m-1)+2}+O\bigl(q^{2a(m-1)+3}\bigr).
\]
Hence, for $F=\Phi^k$ and every $d\ge1$,
\[
W_d-T=k(1-k)^{d-1}\,q^{2a(d-1)+2}+O\bigl(q^{2a(d-1)+3}\bigr),
\]
and over $\ZZ$ the leading coefficient is nonzero when $k\ge2$.
\end{theorem}

\begin{proof}
By Lemma \ref{lem:reduced}, $\cN=-f_1q^{1-a}\,Y+\cN'$ with $\cN'\in q^{2-a}R[[q]][[Y]]$, and $1/\cD$ differs from $1/(1-(1+f_1)Y)$ by an element of $Q\,R[[q]][[Y]]$. So the coefficient of $Y^m$ in $\cN/\cD$ is $-f_1(1+f_1)^{m-1}q^{1-a}+O(q^{2-a})$, and part (iii) of the lemma gives the formula for $c_m$. In $W_d-T=c_d+c_{2d}+\cdots$ the terms $c_{jd}$ with $j\ge2$ have order at least $2a(2d-1)+2$. Finally $f_1=-k$ for $F=\Phi^k$.
\end{proof}

For $k=1$ we have $1+f_1=0$, and the leading coefficient in Theorem \ref{thm:contact} vanishes for $d\ge2$. The order then depends on the parity of $d$.

\begin{theorem}\label{thm:parity}
Suppose that $f_1=-1$. Then for $\ell\ge0$,
\[
c_{2\ell+1}=(-f_2)^{\ell}\,q^{5a\ell+2}+O\bigl(q^{5a\ell+3}\bigr),\qquad
c_{2\ell+2}=f_2(-f_2)^{\ell}\,q^{5a\ell+2a+3}+O\bigl(q^{5a\ell+2a+4}\bigr).
\]
In particular, for $F=\Phi$ and $d\ge1$,
\[
W_d-T=(-1)^{d-1}q^{E_d}+O\bigl(q^{E_d+1}\bigr),\qquad
E_d=\begin{cases}\tfrac52a(d-1)+2,& d\ \text{odd},\\[2pt] \tfrac12a(5d-6)+3,& d\ \text{even}.\end{cases}
\]
\end{theorem}

\begin{proof}
Now $\cD=1+\mathcal{E}$ with $\mathcal{E}=\sum_{i\ge1}D_iQ^i$, and $1/\cD=\sum_{r\ge0}(-1)^r\mathcal{E}^r$. Consider a product of $r$ terms $D_iQ^i$ of $\mathcal E$, with indices $i_1,\dots,i_r$. It contributes to the power $Q^e$ with $e=i_1+\cdots+i_r\ge r$, and by Lemma \ref{lem:reduced}(i) its degree in $Y$ is at most $e+r\le2e$. The degree $2e$ is reached only for $r=e$ and $i_1=\cdots=i_r=1$, with the coefficient $(-1)^ef_2^{\,e}$. Therefore
\begin{equation}\label{eq:mu}
\frac1{\cD}=\sum_{h\ge0}\mu_hY^h,\qquad \mu_h=O\bigl(Q^{\lceil h/2\rceil}\bigr),\qquad \mu_{2\ell}=(-f_2)^{\ell}Q^{\ell}+O\bigl(Q^{\ell+1}\bigr).
\end{equation}
The coefficient of $Y^m$ in $\cN/\cD$ is $\sum_{j=0}^{m}\nu_j\mu_{m-j}$, and we use Lemma \ref{lem:reduced}(ii).

Let $m=2\ell+1$. The term $j=1$ is $(-f_2)^{\ell}q^{a\ell+1-a}$ up to higher order. The term $j=0$ has order at least $a(\ell+1)$. For $j\ge2$ the order of $\nu_j\mu_{m-j}$ is at least $a(j-3)+2+a\lceil(2\ell+1-j)/2\rceil\ge a\ell+2-a$. So the coefficient is $(-f_2)^{\ell}q^{a\ell+1-a}+O(q^{a\ell+2-a})$, and multiplying by $q^{1-a}Q^{2m}$ gives the formula for $c_{2\ell+1}$.

Let $m=2\ell+2$. The term $j=2$ is $f_2(-f_2)^{\ell}q^{a\ell+2-a}$ up to higher order. The terms $j=0$ and $j=1$ have order at least $a(\ell+1)$ and $a\ell+1$. For $j\ge3$ the order is at least $a(j-3)+2+a\lceil(2\ell+2-j)/2\rceil\ge a\ell+2$. All these numbers exceed $a\ell+2-a$. Multiplying by $q^{1-a}Q^{2m}$ gives the formula for $c_{2\ell+2}$.

The exponents $5a\ell+2$ and $5a\ell+2a+3$ are $E_m$ for $m=2\ell+1$ and $m=2\ell+2$, and $E_m$ increases with $m$. Thus $W_d-T=c_d+O(q^{E_d+1})$, including when the indicated coefficient vanishes in $R$. For $F=\Phi$ we have $f_2=-1$.
\end{proof}

We next record the divisibility of the coefficients of $W_d-T$ over $\ZZ$.

\begin{theorem}\label{thm:stair}
Let $R=\ZZ$, $k\ge3$, $F=\Phi^k$ and $d\ge2$, and let $E_d$ be as in Theorem \ref{thm:parity}.
\begin{enumerate}
\item For every integer $\rho\ge0$ the coefficient of $q^{2a(d-1)+2+\rho}$ in $W_d-T$ is divisible by $(k-1)^{e(\rho)}$, where
\[
e(\rho)=\max\Bigl\{0,\ d-2-\Bigl\lfloor\frac{\rho}{a}\Bigr\rfloor-\Bigl\lfloor\frac{\rho-1}{a}\Bigr\rfloor\Bigr\} .
\]
\item Every coefficient of $W_d-T$ below $q^{E_d}$ is divisible by $k-1$, and the coefficient of $q^{E_d}$ is congruent to $(-1)^{d-1}b^{\lfloor d/2\rfloor}$ modulo $k-1$, where $b=-k(k-3)/2$. This residue is not zero when $k\ge4$.
\end{enumerate}
\end{theorem}

\begin{proof}
(ii) Put $h=k-1$. Modulo $h$ we have $f_1=-k\equiv-1$ and $f_2=-b$. So Theorem \ref{thm:parity}, applied over the ring $\ZZ/h\ZZ$, gives $W_d-T\equiv(-1)^{d-1}b^{\lfloor d/2\rfloor}q^{E_d}+O(q^{E_d+1})$. If $h$ is odd, then $b\equiv1$. If $h$ is even, then $b\equiv1+h/2$, and the powers of $1+h/2$ are congruent to $1$ or to $1+h/2$ modulo $h$. None of these residues is $0$ modulo $h$ when $h\ge3$.

(i) Let $P(Y)=1+(k-1)Y$. By Lemma \ref{lem:reduced}(i), $\cD=P+\mathcal{E}$ with $\mathcal{E}=\sum_{i\ge1}D_iQ^i$, and
\[
\frac{1}{\cD}=\sum_{r\ge0}(-1)^r\frac{\mathcal{E}^r}{P^{r+1}} .
\]
As in the proof of Theorem \ref{thm:parity}, the coefficient of $Q^i$ in $\mathcal{E}^r$ is a polynomial in $Y$ of degree at most $i+r$, and it vanishes for $r>i$. So the coefficient of $Q^i$ in $1/\cD$ is $H_i(Y)/P(Y)^{i+1}$ with $H_i\in\ZZ[Y]$ and $\deg H_i\le 2i$. Write $\cN=\sum_{\rho\ge0}N_\rho(Y)\,q^{\rho+1-a}$. By Lemma \ref{lem:reduced}(ii) a power $Y^j$ with $j\ge2$ can occur in $N_\rho$ only if $\rho+1-a\ge a(j-3)+2$. So $N_\rho$ is a polynomial of degree at most $\lfloor(\rho-1)/a\rfloor+2$. Hence the coefficient of $q^{\rho+1-a}$ in $\cN/\cD$ is
\[
\sum_{ai\le\rho}N_{\rho-ai}\,\frac{H_i}{P^{i+1}}=\frac{\mathcal{H}_\rho(Y)}{P(Y)^{\lfloor\rho/a\rfloor+1}},\qquad \mathcal H_\rho\in\ZZ[Y],\quad \deg\mathcal{H}_\rho\le\Bigl\lfloor\frac{\rho}{a}\Bigr\rfloor+\Bigl\lfloor\frac{\rho-1}{a}\Bigr\rfloor+2 .
\]
For $s\ge1$ the coefficient of $Y^m$ in $P^{-s}$ is $\binom{m+s-1}{s-1}(1-k)^m$. So the coefficient of $Y^d$ in $\mathcal H_\rho/P^{s}$ is divisible by $(k-1)^{d-\deg\mathcal H_\rho}$ if $d\ge\deg\mathcal H_\rho$, and therefore by $(k-1)^{e(\rho)}$. By Lemma \ref{lem:reduced}(iii) this number is the coefficient of $q^{2a(d-1)+2+\rho}$ in $c_d$. If $e(\rho)>0$ then $\rho<ad$, and $c_{2d},c_{3d},\dots$ have order at least $2a(2d-1)+2>2a(d-1)+2+ad$. So they do not contribute to these coefficients.
\end{proof}

\section{The Laplace transform and Poisson summation}\label{sec:laplace}

From now on $0<q<1$, and $Q=e^{-t}$, $q=e^{-t/a}$ with $t>0$. Put
\[
G(v)=e^{-v}\,\Phi(e^{-v})^k\quad(v>0),\qquad G(v)=0\quad(v\le 0).
\]
We also write $G(v)=e^{-v}\Phi(e^{-v})^k$ for its holomorphic extension
from $(0,\infty)$ to $\re v>0$. For $0<x\le1$ let
\begin{equation}\label{eq:Ax}
A_x(z)=\sum_{m\ge 0}z^m\,G\bigl(t(m+x)\bigr)=\sum_{n\ge 1}\frac{f_{n-1}\,e^{-ntx}}{1-z\,e^{-nt}} .
\end{equation}
The first series converges for $|z|<e^{t}$, the second for all $z\ne e^{nt}$, and they agree by the computation in the proof of Proposition \ref{prop:lambert}. Comparing with \eqref{eq:J} we see that
\[
N=A_{1/a},\qquad D=A_1,\qquad J=\frac{A_{1/a}}{A_1}.
\]
So the dilation $a$ enters only through the shift $x=1/a$ of the sampling points.

\begin{lemma}\label{lem:inversion}
For $\re v>0$,
\begin{equation}\label{eq:inversion}
G(v)=\Bigl(\frac{2\pi}{v}\Bigr)^{k/2}\exp\Bigl[-\Bigl(1-\frac{k}{24}\Bigr)v-\frac{\pi^2k}{6v}\Bigr]\,\Phi\bigl(e^{-4\pi^2/v}\bigr)^k .
\end{equation}
Hence $G$ is infinitely differentiable on $\RR$, and $G$ and all its derivatives are $O(e^{-v})$ as $v\to+\infty$.
\end{lemma}

\begin{proof}
This is the transformation formula $\eta(-1/\tau)=\sqrt{-i\tau}\,\eta(\tau)$ for $\eta(\tau)=e^{\pi i\tau/12}\Phi(e^{2\pi i\tau})$ \cite[Theorem 3.1]{Apostol}, with $\tau=iv/(2\pi)$. On $|u-v|\le v/2$ with $v>0$ sufficiently small, $\re(1/u)\ge 2/(9v)$,
so \eqref{eq:inversion} bounds $|G(u)|$ by $C v^{-k/2}e^{-c/v}$ for some $c>0$.
Cauchy's inequalities give $G^{(m)}(v)=O_m(v^{-k/2-m}e^{-c/v})$ for every $m\ge0$.
Thus the extension by zero is smooth at the origin. For $v\ge1$ the series $G(v)=\sum_{n\ge0}f_ne^{-(n+1)v}$ may be differentiated term by term.
\end{proof}

\begin{proposition}\label{prop:laplace}
Let $\LL(s)=\int_0^\infty e^{-sv}G(v)\,dv=\int_0^1x^{s}\,\Phi(x)^k\,dx$ for $\re s>-1$.
\begin{enumerate}
\item $\LL$ extends to a meromorphic function on $\CC$. Its poles are the points $s=-n-1$ with $f_n\ne0$. They are simple, with residue $f_n$.
\item Put $y=\sqrt{24s+24-k}$. For every $s$ with $24s+24-k\notin(-\infty,0]$,
\begin{equation}\label{eq:bessel}
\LL(s)=2(2\pi)^{k/2}\sum_{n\ge 0}f_n\Bigl(\frac{y}{2\pi\sqrt{k+24n}}\Bigr)^{\frac k2-1}K_{\frac k2-1}\Bigl(\frac{\pi\,y\sqrt{k+24n}}{6}\Bigr),
\end{equation}
where $K_\nu$ is the modified Bessel function.
\item For $k=1$ and $k=3$, and for all $s$,
\begin{align}
k=1:&\qquad \LL(s)=\frac{4\sqrt3\,\pi}{y}\,\frac{\sinh(\pi y/3)}{\cosh(\pi y/2)},\qquad y=\sqrt{24s+23},\label{eq:L1}\\
k=3:&\qquad \LL(s)=2\pi\,\sech\Bigl(\frac{\pi}{2}\sqrt{7+8s}\Bigr).\label{eq:L3}
\end{align}
\end{enumerate}
\end{proposition}

\begin{proof}
(i) For $\re s>-1$ we split the integral at $v=1$. The part over $(0,1)$ is an entire function of $s$. In the part over $(1,\infty)$ we insert $G(v)=\sum_nf_ne^{-(n+1)v}$ and obtain $\sum_nf_n\,e^{-(s+n+1)}/(s+n+1)$. Since $\sum_n|f_n|e^{-n}<\infty$, this series converges locally uniformly on $\CC\setminus\{-1,-2,\dots\}$.

(ii) Expand $\Phi(e^{-4\pi^2/v})^k=\sum_nf_ne^{-4\pi^2n/v}$ in \eqref{eq:inversion}. With $p=s+1-\frac{k}{24}$ and $B_n=\pi^2(k+24n)/6$ this gives $e^{-sv}G(v)=(2\pi)^{k/2}\sum_nf_n\,v^{-k/2}e^{-pv-B_n/v}$. Let first $p>0$. A change of variable in the integral representation \cite[(10.32.10)]{DLMF} of $K_\nu$ gives
\[
\int_0^\infty v^{-\nu-1}e^{-pv-B/v}\,dv=2\Bigl(\frac{p}{B}\Bigr)^{\nu/2}K_\nu\bigl(2\sqrt{Bp}\bigr)\qquad(p,B>0).
\]
We take $\nu=\frac k2-1$ and integrate term by term. Since $p=y^2/24$, the result is \eqref{eq:bessel}. The exchange of sum and integral is justified because $f_n=O(n^{k/4})$, while $K_\nu(u)$ decays like $e^{-u}$. To see the bound for $f_n$, note that the function $|\eta(\tau)|^4\im\tau$ is invariant under the modular group and bounded on the upper half plane. So $|\Phi(x)|^k\le C\,n^{k/4}$ on the circle $|x|=e^{-1/n}$, and Cauchy's inequality applies. By the same bounds the right side of \eqref{eq:bessel} converges locally uniformly for $\re y>0$, that is, for $s$ outside the ray $24s+24-k\le0$. All the poles of $\LL$ lie on this ray. The identity theorem extends \eqref{eq:bessel} to its connected complement.

(iii) For $k=1$ and $k=3$ we have $K_{\pm1/2}(u)=\sqrt{\pi/(2u)}\,e^{-u}$. By Euler's pentagonal number theorem, for $k=1$ the nonzero $f_n$ are $(-1)^m$ at $1+24n=(6m-1)^2$, $m\in\ZZ$, and \eqref{eq:bessel} becomes $\frac{4\sqrt3\pi}{y}\sum_{m\in\ZZ}(-1)^me^{-|6m-1|\pi y/6}$, which is \eqref{eq:L1}. By Jacobi's identity, for $k=3$ the nonzero $f_n$ are $(-1)^j(2j+1)$ at $3+24n=3(2j+1)^2$, $j\ge0$, and \eqref{eq:bessel} becomes $4\pi\sum_{j\ge0}(-1)^je^{-(2j+1)u}$ with $u=\frac{\pi}{2}\sqrt{7+8s}$, which is \eqref{eq:L3}. The closed forms are even functions of $y$. So they are meromorphic in $s$, and they hold for all $s$.
\end{proof}

Since $\eta(ix)^k=e^{-kv/24}\Phi(e^{-v})^k$ with $v=2\pi x$, formulas \eqref{eq:L1} and \eqref{eq:L3} can be written as
\[
\int_0^\infty e^{-px}\,\eta(ix)\,dx=\sqrt{\frac{\pi}{p}}\;\frac{\sinh\bigl(2\sqrt{\pi p/3}\bigr)}{\cosh\sqrt{3\pi p}},\qquad
\int_0^\infty e^{-px}\,\eta(ix)^3\,dx=\sech\sqrt{\pi p}\qquad(p>0).
\]
These are formulas (5) and (14) of Glasser \cite{Glasser}.

\begin{lemma}\label{lem:size}
Fix $0<\theta<\pi/2$. As $|y|\to\infty$ in the sector $|\arg y|\le\frac\pi2-\theta$,
\begin{equation}\label{eq:Lasym}
\LL(s)=4\sqrt3\,\pi\,k^{(1-k)/4}\;y^{(k-3)/2}\,e^{-\pi\sqrt k\,y/6}\,\bigl(1+O(|y|^{-1})\bigr).
\end{equation}
Consequently, with $\kappa=\pi\sqrt{2k/3}$, there are positive constants $\xi_0,C_1,C_2$ such that
\begin{equation}\label{eq:Lsize}
C_1|\xi|^{(k-3)/4}e^{-\kappa\sqrt{\pi|\xi|}}\le\bigl|\LL(2\pi i\xi)\bigr|\le C_2|\xi|^{(k-3)/4}e^{-\kappa\sqrt{\pi|\xi|}}\qquad(\xi\in\RR,\ |\xi|\ge\xi_0).
\end{equation}
\end{lemma}

\begin{proof}
In \eqref{eq:bessel} we use $K_\nu(u)=\sqrt{\pi/(2u)}\,e^{-u}(1+O(1/|u|))$, which holds uniformly in $|\arg u|\le\frac\pi2$ \cite[(10.40.2)]{DLMF}. The term $n=0$ gives the right side of \eqref{eq:Lasym}. The terms with $n\ge1$ are smaller by factors $O\bigl(n^{C}e^{-(\sqrt{k+24n}-\sqrt k)\pi\re y/6}\bigr)$, and $\re y\ge|y|\sin\theta$. For \eqref{eq:Lsize} take $s=2\pi i\xi$. Then $y=\sqrt{24-k+48\pi i\xi}$ satisfies $\arg y\to\pm\frac\pi4$ and $\re y=\sqrt{24\pi|\xi|}+O(|\xi|^{-1/2})$ as $|\xi|\to\infty$, so that $\frac{\pi\sqrt k}{6}\re y=\kappa\sqrt{\pi|\xi|}+O(|\xi|^{-1/2})$.
\end{proof}

\begin{proposition}\label{prop:dual}
Let $0<x\le1$. As an identity between meromorphic functions of $w\in\CC$,
\begin{equation}\label{eq:dualA}
A_x\bigl(e^{w}\bigr)=\frac1t\sum_{n\in\ZZ}e^{(2\pi in-w)x}\,\LL\Bigl(\frac{2\pi in-w}{t}\Bigr).
\end{equation}
Consequently
\begin{equation}\label{eq:dualJ}
J\bigl(e^{w}\bigr)=e^{(1-1/a)w}\;\frac{\sum_{n\in\ZZ}e^{2\pi in/a}\,\LL\bigl((2\pi in-w)/t\bigr)}{\sum_{n\in\ZZ}\LL\bigl((2\pi in-w)/t\bigr)} .
\end{equation}
\end{proposition}

\begin{proof}
Let first $\re w<t$. By Lemma \ref{lem:inversion} the function $g(u)=e^{wu}G(t(u+x))$ is a Schwartz function on $\RR$, and $g(m)=0$ for all integers $m\le-1$ because $x\le1$. So $\sum_{m\in\ZZ}g(m)=A_x(e^w)$. The Fourier transform of $g$ at an integer $n$ is
\[
\int_{\RR}g(u)\,e^{-2\pi inu}\,du=\frac1t\,e^{(2\pi in-w)x}\,\LL\Bigl(\frac{2\pi in-w}{t}\Bigr),
\]
as the substitution $v=t(u+x)$ shows. So \eqref{eq:dualA} is the Poisson summation formula for $g$ \cite[Chapter 1]{GroBook}. Now let $w$ be arbitrary. For $s=(2\pi in-w)/t$ and $|n|\to\infty$ the number $y=\sqrt{24s+24-k}$ satisfies $\arg y\to\pm\frac\pi4$ and $|y|\to\infty$, locally uniformly in $w$. By \eqref{eq:Lasym} the terms of the series are then $O(e^{-\epsilon\sqrt{|n|}})$ for some $\epsilon>0$, locally uniformly in $w$. Hence the right side of \eqref{eq:dualA} is a meromorphic function of $w\in\CC$. So is the left side, by \eqref{eq:Ax}. Since the two sides agree for $\re w<t$, they agree everywhere. Dividing the cases $x=1/a$ and $x=1$ gives \eqref{eq:dualJ}.
\end{proof}

Formula \eqref{eq:dualJ} makes the interpolation property visible. Choose the real logarithm $w=mt$ at a node $e^{mt}$ with $f_{m-1}\ne0$. The terms with $n=0$ in the numerator and denominator then have the same principal part, so the quotient of the two sums has removable value $1$. In the next section we show that for $w=\log z$ with $z\in\Omega$, the set in Theorem \ref{thm:radial}, these terms dominate as $t\downarrow0$.

At $k=3$, substituting $x=1$, $w=0$ and \eqref{eq:L3} into \eqref{eq:dualA} gives, for $v=e^{-t}$,
\[
 \sum_{m\ge1}v^m\Phi(v^m)^3=
 \frac{2\pi}{t}\sum_{n\in\ZZ}\sech\!\left(\frac{\pi}{2}\sqrt{7+\frac{16\pi i n}{t}}\right).
\]

\section{The limit \texorpdfstring{$q\to1$}{q to 1}}\label{sec:radial}

Recall that $\beta=1-1/a$ and $\kappa=\pi\sqrt{2k/3}$. Let $\Omega=\CC\setminus\bigl((-\infty,0]\cup(1,\infty)\bigr)$. For $z\in\CC\setminus(-\infty,0]$ let $w=\log z$, so that $|\im w|<\pi$, and put
\begin{equation}\label{eq:cz}
c(z)=\kappa\,\Bigl[\min\bigl\{\re\sqrt{2\pi i-w},\ \re\sqrt{-2\pi i-w}\bigr\}-\re\sqrt{-w}\,\Bigr].
\end{equation}
The real part of $\sqrt{A+iB}$ increases strictly with $|B|$. Since $|\im w|<\pi$, this shows that $c(z)>0$. The function $c$ is continuous, and on the unit circle $c(e^{2\pi i\alpha})=\kappa\sqrt\pi\,\bigl(\sqrt{1-|\alpha|}-\sqrt{|\alpha|}\bigr)$ for $|\alpha|<\frac12$.

\begin{theorem}\label{thm:radial}
Let $k\ge1$, let $z\in\Omega$ and $c<c(z)$. Then
\[
J\bigl(z,e^{-t/a}\bigr)=z^{\beta}\bigl(1+O(e^{-c/\sqrt t})\bigr)\qquad(t\to0^+).
\]
If $K$ is a compact subset of $\Omega$ which does not contain the point $1$, and if $c<\min_Kc(z)$, then the estimate holds uniformly for $z\in K$.
\end{theorem}

\begin{proof}
Let $w=\log z$ and $s_n=(2\pi in-w)/t$ for $n\in\ZZ$. By \eqref{eq:dualA},
\begin{equation}\label{eq:Rx}
t\,A_x(z)=e^{-wx}\LL(s_0)\,(1+R_x),\qquad
R_x=\sum_{n\ne0}e^{2\pi inx}\,\frac{\LL(s_n)}{\LL(s_0)},
\end{equation}
provided that $\LL(s_0)$ is finite and not zero.

Let $z\ne1$. Then no number $2\pi in-w$ lies in $(-\infty,0]$: for $n\ne0$ it is not real, and for $n=0$ this holds because $w\notin[0,\infty)$. Put $r_n=\re\sqrt{2\pi in-w}>0$ and $y_n=\sqrt{24s_n+24-k}$. For small $t$ we have $y_n=\sqrt{24/t}\,\sqrt{2\pi in-w}\,\bigl(1+O(t/|2\pi in-w|)\bigr)$. So all $y_n$ lie in a sector $|\arg y|\le\frac\pi2-\theta$, they tend to infinity, and
\[
\frac{\pi\sqrt k}{6}\,\re y_n=\frac{\kappa\,r_n}{\sqrt t}+O(\sqrt t)
\]
uniformly in $n$. By \eqref{eq:Lasym}, $\LL(s_0)$ is finite and not zero for small $t$, and
\[
\Bigl|\frac{\LL(s_n)}{\LL(s_0)}\Bigr|\le C\,\Bigl|\frac{y_n}{y_0}\Bigr|^{(k-3)/2}\exp\Bigl[-\frac{\kappa\,(r_n-r_0)}{\sqrt t}\Bigr]\le C'\,(1+|n|)^{|k-3|/4}\exp\Bigl[-\frac{\kappa\,(r_n-r_0)}{\sqrt t}\Bigr].
\]
The number $r_n$ increases with $|2\pi n-\im w|$. So $r_n\ge r_*:=\min(r_1,r_{-1})$ for $n\ne0$, and $r_n$ grows like $\sqrt{\pi|n|}$. Since $\kappa(r_*-r_0)=c(z)$, we get for $t\le1$
\[
|R_x|\le C'\,e^{-c(z)/\sqrt t}\sum_{n\ne0}(1+|n|)^{|k-3|/4}e^{-\kappa(r_n-r_*)}=C''\,e^{-c(z)/\sqrt t}.
\]
All constants can be chosen uniformly for $z$ in a compact subset of $\Omega\setminus\{1\}$.

Let $z=1$. Then $w=0$ and $\LL(s_0)=\LL(0)>0$, and \eqref{eq:Lsize} with $\xi=n/t$ gives
\[
|R_x|\le C\sum_{n\ne0}\Bigl(\frac{|n|}{t}\Bigr)^{(k-3)/4}e^{-\kappa\sqrt{\pi|n|/t}}=O\bigl(e^{-c/\sqrt t}\bigr)\qquad\text{for every } c<\kappa\sqrt\pi=c(1).
\]
In both cases $J(z)=A_{1/a}(z)/A_1(z)=e^{w(1-1/a)}(1+R_{1/a})/(1+R_1)$, and $e^{w(1-1/a)}=z^\beta$.
\end{proof}

The rate $c(z)$ tends to zero as $z$ approaches any point of $(-\infty,0]$.
At each fixed $x>1$, however, $c(z)\to c(x)>0$ as $z\to x$. There $s_0=-(\log x)/t$ lies
on the negative real axis, where $\LL$ has poles. Between consecutive poles with residues of
the same sign it also has a zero, by the sign argument of Remark~\ref{rem:bary}. For $k=3$ the transform $\LL$ has no zeros, and the ray $(1,\infty)$ need not be excluded (Theorem \ref{thm:limit}).

\begin{corollary}\label{cor:odd}
Let $k\ge1$ and let $d$ be odd. Then, as $t\to0^+$,
\[
W_d\bigl(e^{-t/a}\bigr)=\frac{\sin(\pi/a)}{d\,\sin\bigl(\pi(a-1)/(ad)\bigr)}+O\bigl(e^{-c/\sqrt t}\bigr)\qquad\text{for every } c<\kappa\sqrt\pi\,\Bigl(\sqrt{\tfrac{d+1}{2d}}-\sqrt{\tfrac{d-1}{2d}}\Bigr).
\]
\end{corollary}

\begin{proof}
For odd $d$ the $d$-th roots of unity are the points $e^{2\pi ir/d}$ with integers $|r|\le(d-1)/2$. They lie in $\Omega$. By Theorem \ref{thm:radial} and the formula for $c$ on the unit circle, $W_d$ differs from
\[
\frac1d\sum_{|r|\le(d-1)/2}e^{2\pi ir\beta/d}=\frac{\sin(\pi\beta)}{d\,\sin(\pi\beta/d)}
\]
by $O(e^{-c/\sqrt t})$, and $\sin(\pi\beta)=\sin(\pi/a)$.
\end{proof}

The point $z=-1$ lies on the cut, and there $J$ oscillates. Put
\[
\Theta(t)=\pi^{3/2}\sqrt{\frac{k}{3t}},\qquad \psi(t)=\Theta(t)-\frac{(k-3)\pi}{8},\qquad M(t)=\frac{8\sqrt3\,\pi}{t}\,k^{(1-k)/4}\Bigl(\frac{24\pi}{t}\Bigr)^{(k-3)/4}e^{-\Theta(t)} .
\]

\begin{theorem}\label{thm:cos}
Let $k\ge1$ and $0<x\le1$. As $t\to0^+$,
\[
A_x(-1)=M(t)\bigl\{\cos\bigl(\psi(t)-\pi x\bigr)+\epsilon_x(t)\bigr\},\qquad \epsilon_x(t)=O(\sqrt t),\quad \epsilon_x'(t)=O(1/t).
\]
\end{theorem}

\begin{proof}
Take $w=i\pi$ in \eqref{eq:dualA}. Since $\LL(\bar s)=\overline{\LL(s)}$, the terms $n$ and $1-n$ are complex conjugates, and
\begin{equation}\label{eq:Aminus}
A_x(-1)=\frac2t\,\re\sum_{\substack{m\ge1\\ m\ \mathrm{odd}}}e^{i\pi mx}\,\LL\Bigl(\frac{i\pi m}{t}\Bigr).
\end{equation}
We study $\LL(i\pi m/t)$ for complex $t$ in the sector $S=\{|\arg t|\le\frac\pi8,\ 0<|t|\le t_0\}$, and we extend $\Theta$, $\psi$ and $M$ analytically to $S$. For $s=i\pi m/t$ we have $y=\sqrt{24\pi im/t}\,\bigl(1+O(|t|/m)\bigr)$. So $|\arg y|\le\frac{3\pi}8$ if $t_0$ is small, and \eqref{eq:Lasym} applies. Since $\frac{\pi\sqrt k}{6}\sqrt{24\pi i/t}=(1+i)\,\Theta(t)$, we obtain
\[
\frac{\pi\sqrt k}{6}\,y=(1+i)\sqrt m\;\Theta(t)+O\bigl(|t|^{1/2}\bigr),\qquad
y^{(k-3)/2}=\Bigl(\frac{24\pi m}{t}\Bigr)^{(k-3)/4}e^{i(k-3)\pi/8}\,\bigl(1+O(|t|)\bigr),
\]
uniformly in $m$. For $m=1$ this gives
\begin{equation}\label{eq:Lpi}
\LL\Bigl(\frac{i\pi}{t}\Bigr)=\frac t2\,M(t)\,e^{-i\psi(t)}\bigl(1+\varepsilon(t)\bigr),
\end{equation}
where $\varepsilon$ is analytic in $S$ and $\varepsilon(t)=O(|t|^{1/2})$. For odd $m\ge3$ and real $t$ it gives
\begin{equation}\label{eq:Ltail}
\Bigl|\LL\Bigl(\frac{i\pi m}{t}\Bigr)\Bigr|\le C\,t\,M(t)\,m^{|k-3|/4}\,e^{-(\sqrt m-1)\Theta(t)} .
\end{equation}
Inserting \eqref{eq:Lpi} in \eqref{eq:Aminus} we get the asserted formula with
\[
\epsilon_x(t)=\re\bigl[e^{-i(\psi(t)-\pi x)}\varepsilon(t)\bigr]+\re\,\widetilde\tau_x(t),\qquad
\widetilde\tau_x(t)=\frac{2}{t\,M(t)}\sum_{\substack{m\ge3\\ m\ \mathrm{odd}}}e^{i\pi mx}\,\LL\Bigl(\frac{i\pi m}{t}\Bigr),
\]
and $\widetilde\tau_x(t)=O\bigl(e^{-(\sqrt3-1)\Theta(t)}\bigr)$ by \eqref{eq:Ltail}. So $\epsilon_x(t)=O(\sqrt t)$.

For $0<t<t_0/2$, the disc $|t'-t|\le t\sin\frac\pi8$ lies in $S$, so $\varepsilon'(t)=O(t^{-1/2})$. Since $\psi'(t)=-\Theta(t)/(2t)=O(t^{-3/2})$, the derivative of $e^{-i(\psi-\pi x)}\varepsilon$ is $O(1/t)$. The function $\widetilde\tau_x$ is analytic in $S$. On the disc $|t'-t|\le t^{3/2}$ we have $\Theta(t')=\Theta(t)+O(1)$. Hence $|M(t')/M(t)|$ is bounded above and below there, and \eqref{eq:Lasym}
gives \eqref{eq:Ltail} at $t'$ with an additional factor $e^{C\sqrt m}$.
For small $t$, $(\sqrt m-1)\Theta-C\sqrt m\ge\Theta/2+\sqrt m$
for every $m\ge3$. Summation gives $|\widetilde\tau_x(t')|=O(e^{-\Theta(t)/2})$ on this disc, and $\widetilde\tau_x'(t)=O\bigl(t^{-3/2}e^{-\Theta(t)/2}\bigr)=O(1)$. Together this gives $\epsilon_x'(t)=O(1/t)$.
\end{proof}

\begin{theorem}\label{thm:poles}
Let $k\ge1$ and let $d$ be even.
\begin{enumerate}
\item For each fixed $\delta>0$, uniformly for small $t$ with $|\cos\psi(t)|\ge\delta$,
\[
W_d\bigl(e^{-t/a}\bigr)=\frac{\sin(\pi/a)}{d}\Bigl\{\cot\frac{\pi(a-1)}{ad}-\tan\psi(t)\Bigr\}+O(\sqrt t).
\]
\item There are $j_0\ge1$ and $t_0>0$ such that the zeros of the function $t\mapsto D(-1,e^{-t/a})$ in $(0,t_0]$ are simple and can be labelled $t_j$, $j\ge j_0$, so that $t_j$ decreases to $0$ and
\[
\Theta(t_j)=\pi\Bigl(j+\frac{k+1}{8}\Bigr)+\frac{c_k}{j}+O\Bigl(\frac1{j^2}\Bigr),\qquad c_k=\frac{\pi\,k(24-k)}{144}-\frac{(k-1)(k-3)}{16\pi}.
\]
In particular $t_j=\pi k\big/\bigl(3(j+\frac{k+1}{8})^2\bigr)\cdot\bigl(1+O(j^{-2})\bigr)$.
\item The poles of $W_d$ in the interval $[e^{-t_0/a},1)$ are exactly the points $q_j=e^{-t_j/a}$. They are simple, and
\[
\Res_{q=q_j}W_d=\frac{2\,q_j\,t_j\,\sin(\pi/a)}{a\,d\;\Theta(t_j)}\Bigl(1+O\Bigl(\frac1j\Bigr)\Bigr).
\]
In particular $1-q_j\sim\pi k/(3aj^2)$ and $\Res_{q=q_j}W_d\sim 2k\sin(\pi/a)/(3adj^3)$, and the residues are positive for large $j$.
\end{enumerate}
\end{theorem}

\begin{proof}
(i) The $d$-th roots of unity other than $-1$ are the points $e^{2\pi ir/d}$ with $|r|<d/2$. By Theorem \ref{thm:radial} they contribute
\[
\frac1d\sum_{|r|<d/2}e^{2\pi ir\beta/d}+O\bigl(e^{-c/\sqrt t}\bigr)=\frac1d\Bigl\{\sin(\pi\beta)\cot\frac{\pi\beta}{d}-\cos(\pi\beta)\Bigr\}+O\bigl(e^{-c/\sqrt t}\bigr)
\]
to $W_d$. By Theorem \ref{thm:cos} with $x=1/a$ and $x=1$,
\[
J(-1)=\frac{A_{1/a}(-1)}{A_1(-1)}=-\frac{\cos(\psi-\pi/a)+\epsilon_{1/a}}{\cos\psi-\epsilon_1}=-\cos\frac\pi a-\sin\frac\pi a\,\tan\psi+O(\sqrt t)
\]
when $|\cos\psi|\ge\delta$. Since $\cos(\pi/a)=-\cos(\pi\beta)$ and $\sin(\pi/a)=\sin(\pi\beta)$, adding the two parts gives the claim.

(ii) By Theorem \ref{thm:cos} with $x=1$, $A_1(-1)=-M(t)\{\cos\psi(t)-\epsilon_1(t)\}$. Take $t_0$ sufficiently small for the estimates below, with $\psi(t_0)=\pi j_0$
for some integer $j_0$. In particular $|\epsilon_1|<1/\sqrt2$ on $(0,t_0]$. Then there is no zero where $|\cos\psi|\ge1/\sqrt2$. The remaining $t$ lie in the intervals $|\psi(t)-\pi(j+\frac12)|<\frac\pi4$, $j\ge j_0$. On each interval the derivative of $\cos\psi-\epsilon_1$ is $-\sin\psi\cdot\psi'(t)+O(1/t)$, where $|\sin\psi|\ge1/\sqrt2$ and $|\psi'(t)|=\Theta(t)/(2t)$ is of the order $t^{-3/2}$. Taking $j_0$ sufficiently large, $\cos\psi-\epsilon_1$ is strictly monotone on each interval and has opposite signs at its endpoints. Each interval therefore contains one simple zero $t_j$. At the zero $\cos\psi(t_j)=\epsilon_1(t_j)=O(\sqrt{t_j})=O(1/j)$, which gives $\psi(t_j)=\pi(j+\frac12)+O(1/j)$, that is $\Theta(t_j)=\pi(j+\frac{k+1}8)+O(1/j)$.

For small positive $t$, \eqref{eq:Lpi} gives the continuous phase
$\alpha(t)=-\psi(t)+\im\log(1+\varepsilon(t))$ of $\LL(i\pi/t)$,
where the logarithm is principal. In \eqref{eq:bessel} the term $n=0$ has the argument $u=\frac{\pi\sqrt k}{6}\,y$, and the other terms are smaller by an exponentially small factor. By \cite[(10.40.2)]{DLMF},
\[
K_\nu(u)=\sqrt{\frac{\pi}{2u}}\;e^{-u}\Bigl(1+\frac{4\nu^2-1}{8u}+O\bigl(|u|^{-2}\bigr)\Bigr),\qquad 4\nu^2-1=(k-1)(k-3)\quad\text{for }\nu=\tfrac k2-1 .
\]
For $s=i\pi/t$ we have $y=\sqrt{24\pi/t}\;e^{i\pi/4}\bigl(1-i\,\frac{(24-k)t}{48\pi}+O(t^2)\bigr)$. With $\Theta=\Theta(t)$ and $t\,\Theta=\pi^3k/(3\Theta)$ this gives
\[
u=(1+i)\,\Theta+(1-i)\,\frac{\pi^2k(24-k)}{144\,\Theta}+O(t^{3/2}),\qquad \im u=\Theta-\frac{\pi^2k(24-k)}{144\,\Theta}+O(t^{3/2}),
\]
and $\im(1/u)=-1/(2\Theta)+O(t^{3/2})$. Hence
\[
\alpha(t)=\frac{k-3}{2}\arg y-\im u+\frac{(k-1)(k-3)}{8}\,\im\frac1u+O(t)=-\psi(t)+\frac{\pi c_k}{\Theta(t)}+O(t).
\]
By \eqref{eq:Aminus} and \eqref{eq:Ltail}, $A_1(-1)$ equals $-\frac2t\re\LL(i\pi/t)$ up to a term which is smaller than $\frac2t|\LL(i\pi/t)|$ by an exponentially small factor. So at a zero the phase of $\LL(i\pi/t)$ is an odd multiple of $\frac\pi2$, up to an exponentially small error. With the first approximation this gives $\psi(t_j)-\pi c_k/\Theta(t_j)=\pi(j+\frac12)+O(j^{-2})$, and $\Theta(t_j)=\pi j\,(1+O(1/j))$ gives the claim. The formula for $t_j$ follows from $t=\pi^3k/(3\Theta(t)^2)$.

(iii) By the proof of Theorem \ref{thm:radial}, for small $t$ the denominator $D(\omega)$ does not vanish at the $d$-th roots of unity $\omega\ne-1$. So the poles of $W_d$ come from $J(-1)=A_{1/a}(-1)/A_1(-1)$ alone. At $t=t_j$ we have $\cos\psi=O(1/j)$, hence by Theorem \ref{thm:cos}
\[
A_{1/a}(-1)=M\bigl\{\sin\psi\,\sin\tfrac{\pi}{a}+O(1/j)\bigr\},\qquad \frac{d}{dt}A_1(-1)=M\bigl\{\sin\psi\cdot\psi'(t_j)+O(1/t_j)\bigr\},
\]
with $|\sin\psi|=1+O(j^{-2})$. So the numerator does not vanish for large $j$, the pole is simple, and the quotient of the two expressions is $-\frac{2t_j\sin(\pi/a)}{\Theta(t_j)}(1+O(1/j))$. Since $dq/dt=-q/a$, the residue of $W_d$ in the variable $q$ is $-\frac{q_j}{ad}$ times this number. The last statements follow from $t_j\sim\pi k/(3j^2)$ and $\Theta(t_j)\sim\pi j$.
\end{proof}

\begin{remark}\label{rem:index}
The index $j$ in Theorem~\ref{thm:poles} labels an interval in the variable $\psi(t)$. Relating it to the numbering of all positive real poles from the origin requires the number of earlier crossings.

For integral $a$, each $W_d$ is meromorphic in the unit disc as a function of $q$. For even $d$, the accumulation of poles at $1$ precludes modularity as a function of $\tau$, where $q=e^{2\pi i\tau}$, for any subgroup of finite index of the modular group. Indeed, such a function is meromorphic at the cusp $0$, so it has no poles on the segment $\tau=iy$ with small $y>0$, while the poles $q_j$ accumulate at $q=1$.
\end{remark}

\section{Total positivity and interlacing for \texorpdfstring{$k=3$}{k=3}}\label{sec:tp}

From here to the end of Section \ref{sec:where} we assume $k=3$. By Jacobi's identity, $\Phi(x)^3=\sum_{j\ge0}(-1)^j(2j+1)\,x^{j(j+1)/2}$ \cite{Andrews}, so that
\begin{equation}\label{eq:jacobi}
G(v)=\sum_{j\ge0}\rho_j\,e^{-\lambda_jv}\quad(v>0),\qquad \rho_j=(-1)^j(2j+1),\quad \lambda_j=1+\frac{j(j+1)}{2},
\end{equation}
and \eqref{eq:Ax} reads
\begin{equation}\label{eq:Ax3}
A_x(z)=\sum_{j\ge0}\frac{\rho_j\,e^{-\lambda_jtx}}{1-z\,e^{-\lambda_jt}} .
\end{equation}
The poles of $A_x$ are the nodes $z=e^{\lambda_jt}=Q^{-\lambda_j}$, and they are simple. Since $\cosh(\frac\pi2\sqrt u)=\prod_{j\ge0}\bigl(1+u/(2j+1)^2\bigr)$ and $(2j+1)^2+7=8\lambda_j$, formula \eqref{eq:L3} can be written as
\begin{equation}\label{eq:Lprod}
\LL(s)=2\pi\,\sech\Bigl(\frac\pi2\sqrt{7+8s}\Bigr)=\LL(0)\prod_{j\ge0}\frac{\lambda_j}{\lambda_j+s} .
\end{equation}

An integrable function $\Lambda:\RR\to[0,\infty)$ of positive integral is a \emph{P\'olya frequency function} if
\[
\det\bigl[\Lambda(x_i-y_l)\bigr]_{i,l=1}^{n}\ \ge 0\qquad\text{whenever } x_1<\cdots<x_n,\ \ y_1<\cdots<y_n,\ \ n\ge1 .
\]
We use three classical facts.
\begin{itemize}[leftmargin=2.6em]
\item[(P1)] For $c>0$ the functions $c\,e^{-cv}\mathbf{1}_{v>0}$ and $c\,e^{cv}\mathbf{1}_{v<0}$ are P\'olya frequency functions. Convolution and positive dilations preserve this class. Pointwise limits and multiplication by $e^{\mu v}$ preserve nonnegativity of the minors, and hence give P\'olya frequency functions whenever the result has finite positive integral \cite{Sch51}, \cite[Chapter 7]{Karlin}.
\item[(P2)] If $\Lambda$ is a P\'olya frequency function, $y_1<\cdots<y_n$ and $\gamma_1,\dots,\gamma_n$ are real, then the function $u\mapsto\sum_i\gamma_i\,\Lambda(u-y_i)$ has at most as many sign changes on $\RR$ as the sequence $\gamma_1,\dots,\gamma_n$. This is the variation diminishing property \cite{Sch30,Karlin}.
\item[(P3)] Let $b_0>0$, $b_1,b_2,\dots\ge0$ and $b_m=0$ for $m<0$. If all minors of the matrix $(b_{i-l})_{i,l\ge0}$ are nonnegative, then
\[
\sum_{m\ge0}b_mz^m=b_0\,e^{\gamma_0 z}\,\frac{\prod_i(1+\alpha_iz)}{\prod_i(1-\delta_iz)}
\]
with $\gamma_0\ge0$, $\alpha_i,\delta_i>0$ and $\sum_i(\alpha_i+\delta_i)<\infty$. This is the theorem of Aissen, Schoenberg, Whitney and Edrei \cite{ASW,Edrei}, \cite[Chapter 8]{Karlin}.
\end{itemize}

\begin{lemma}\label{lem:PF}
Let $\ell$ be a nonnegative integer, and let $\sigma$ be real with $-\lambda_\ell<\sigma<-\lambda_{\ell-1}$; for $\ell=0$ the condition is $\sigma>-\lambda_0$. Put
\[
G_\ell(v)=G(v)-\sum_{j<\ell}\rho_j\,e^{-\lambda_jv}\qquad(v\in\RR).
\]
Then $v\mapsto(-1)^{\ell}e^{-\sigma v}G_\ell(v)$ is a continuous P\'olya frequency function, and it is $O(e^{-\epsilon|v|})$ for some $\epsilon>0$. In particular $G=G_0$ is a P\'olya frequency function.
\end{lemma}

\begin{proof}
For $v>0$ we have $G_\ell(v)=\sum_{j\ge\ell}\rho_je^{-\lambda_jv}=O(e^{-\lambda_\ell v})$ by \eqref{eq:jacobi}, while on $v\le0$ the function $G_0$ vanishes and, for $\ell\ge1$, $G_\ell(v)=-\sum_{j<\ell}\rho_je^{-\lambda_jv}=O(e^{\lambda_{\ell-1}|v|})$. Together with $G(0^+)=0$ this gives the continuity of $G_\ell$ and the decay of $e^{-\sigma v}G_\ell(v)$. With $s=\sigma+2\pi i\xi$, the Fourier transform of $e^{-\sigma v}G_\ell(v)$ at $\xi$ is
\begin{multline*}
\int_0^\infty e^{-sv}\Bigl[G(v)-\sum_{j<\ell}\rho_je^{-\lambda_jv}\Bigr]dv-\sum_{j<\ell}\rho_j\int_{-\infty}^0e^{-(s+\lambda_j)v}\,dv\\
=\Bigl[\LL(s)-\sum_{j<\ell}\frac{\rho_j}{s+\lambda_j}\Bigr]+\sum_{j<\ell}\frac{\rho_j}{s+\lambda_j}=\LL(s).
\end{multline*}
Here $\LL$ is the meromorphic function \eqref{eq:Lprod}. The first bracket is the value of the first integral for $\re s>-1$, and then for $\re s>-\lambda_\ell$ by analytic continuation, because the residue of $\LL$ at $-\lambda_j$ is $\rho_j$.

Now let $c_j=\lambda_j+\sigma$, so that $c_j>0$ for $j\ge\ell$ and $c_j<0$ for $j<\ell$. Let $\chi_j(v)=c_je^{-c_jv}\mathbf{1}_{v>0}$ for $j\ge\ell$ and $\chi_j(v)=|c_j|e^{-c_jv}\mathbf{1}_{v<0}$ for $j<\ell$. These are probability densities and P\'olya frequency functions by (P1), and in both cases the Fourier transform of $\chi_j$ is $c_j/(c_j+2\pi i\xi)$. So $\Lambda_n=\chi_0*\cdots*\chi_n$ is a P\'olya frequency function with the Fourier transform $\prod_{j\le n}c_j/(c_j+2\pi i\xi)$. For $n\ge1$ these transforms are bounded by $C/(1+\xi^2)$, and by \eqref{eq:Lprod} they converge to $\LL(\sigma+2\pi i\xi)/\LL(\sigma)$. By Fourier inversion, $\Lambda_n$ converges uniformly to $e^{-\sigma v}G_\ell(v)/\LL(\sigma)$. The limit has integral one by the transform identity at $\xi=0$, so it is a P\'olya frequency function by (P1), and $\LL(\sigma)$ has the sign $(-1)^{\ell}$ by \eqref{eq:Lprod}.
\end{proof}

\begin{theorem}\label{thm:zeros}
Let $k=3$ and $t>0$.
\begin{enumerate}
\item For every $0<x\le1$ the function $A_x$ has infinitely many zeros. They are all negative and simple. Writing them as $0>\zeta_1(x)>\zeta_2(x)>\cdots$, we have
\[
A_x(z)=G(tx)\,\frac{\prod_{i\ge1}\bigl(1-z/\zeta_i(x)\bigr)}{\prod_{j\ge0}\bigl(1-z\,e^{-\lambda_jt}\bigr)},\qquad \sum_{i\ge1}\frac1{|\zeta_i(x)|}<\infty .
\]
\item Each $\zeta_i$ is continuous and strictly decreasing on $(0,1]$, and $\zeta_1(x)\to0$, $\zeta_{i+1}(x)\to\zeta_i(1)$ as $x\to0^+$.
\item For $0<x<1$ the zeros of $A_x$ and of $A_1$ interlace strictly:
\[
0>\zeta_1(x)>\zeta_1(1)>\zeta_2(x)>\zeta_2(1)>\zeta_3(x)>\cdots .
\]
\end{enumerate}
\end{theorem}

\begin{proof}
\emph{Step 1: the zeros are negative.} Let $0<x\le1$ and $b_m=G(t(m+x))$ for $m\ge0$, $b_m=0$ for $m<0$. Since $G(t(i-l+x))=0$ for $i<l$, a minor of $(b_{i-l})$ with rows $i_1<\cdots<i_n$ and columns $l_1<\cdots<l_n$ is $\det[G(x_r-y_s)]$ with $x_r=t(i_r+x)$ and $y_s=tl_s$. It is nonnegative by Lemma \ref{lem:PF}. So (P3) applies to $A_x(z)=\sum_mb_mz^m$. By \eqref{eq:Ax3} the poles of $A_x$ are the points $e^{\lambda_jt}$, and they are simple. Hence the numbers $\delta_i$ in (P3) are the numbers $e^{-\lambda_jt}$, and
\[
P_x(z):=A_x(z)\prod_{j\ge0}\bigl(1-z\,e^{-\lambda_jt}\bigr)=G(tx)\,e^{\gamma_0 z}\prod_i(1+\alpha_iz)
\]
is an entire function. On the other hand \eqref{eq:Ax3} gives $P_x(z)=\sum_j\rho_je^{-\lambda_jtx}\prod_{i\ne j}(1-z\,e^{-\lambda_it})$, so that $|P_x(z)|\le C\prod_i(1+|z|e^{-\lambda_it})$. This product grows more slowly than $e^{\varepsilon|z|}$ for every $\varepsilon>0$. Since $P_x(r)\ge G(tx)e^{\gamma_0 r}$ for $r>0$, we get $\gamma_0=0$. Moreover $P_x$ is not a polynomial. Indeed, at the node $z=e^{\lambda_jt}$ only one term of the sum survives, and
\[
\bigl|P_x(e^{\lambda_jt})\bigr|=(2j+1)e^{-\lambda_jtx}\prod_{i\ne j}\bigl|1-e^{(\lambda_j-\lambda_i)t}\bigr|\ \ge\ e^{-\lambda_jtx}\,\Phi(e^{-t})\,(1-e^{-t})^{j}\exp\Bigl(t\sum_{i<j}(\lambda_j-\lambda_i)\Bigr),
\]
where $\sum_{i<j}(\lambda_j-\lambda_i)=j(j+1)(2j+1)/6$ grows faster than every multiple of $\lambda_j$. This proves the product formula in (i), with $\zeta_i=-1/\alpha_i$, and that there are infinitely many zeros.

\emph{Step 2: the zeros are simple.} Let $0<x\le\frac12$. Step 1, with $t/2$ and $2x$ in place of $t$ and $x$, shows that
\[
f(\tau):=\sum_{n\ge0}G\bigl(\tfrac t2(n+2x)\bigr)\tau^n=f(0)\,\frac{\prod_i(1+\alpha_i\tau)}{\prod_i(1-\delta_i\tau)}
\]
with $\alpha_i,\delta_i>0$ and $\sum_i(\alpha_i+\delta_i)<\infty$. Splitting the sum into even and odd terms gives $f(\tau)=A_x(\tau^2)+\tau A_{x+1/2}(\tau^2)$ near zero and hence by meromorphic continuation. For $\tau=i\eta$ with $\eta>0$ the numbers $A_x(-\eta^2)$ and $A_{x+1/2}(-\eta^2)$ are real, and therefore
\begin{gather*}
A_x(-\eta^2)=|f(i\eta)|\cos\vartheta(\eta),\qquad \eta\,A_{x+1/2}(-\eta^2)=|f(i\eta)|\sin\vartheta(\eta),\\
\vartheta(\eta)=\sum_i\arctan(\alpha_i\eta)+\sum_i\arctan(\delta_i\eta).
\end{gather*}
Here $|f(i\eta)|>0$ and $\vartheta'(\eta)>0$. If $A_x(-\eta_0^2)=0$, then $\cos\vartheta(\eta_0)=0$, and the derivative of $\eta\mapsto A_x(-\eta^2)$ at $\eta_0$ is $-|f|\,\vartheta'\sin\vartheta\ne0$. So the zero is simple. In the same way the zeros of $A_{x+1/2}$ are simple. Every number in $(0,1]$ is of the form $x$ or $x+\frac12$ with $0<x\le\frac12$, and all zeros are negative by Step 1. This completes the proof of (i).

\emph{Step 3: at most one sign change.} Let $r>0$ with $r\ne e^{\lambda_jt}$ for all $j$. We claim that there are no three points $0<\xi_0<\xi_1<\xi_2\le1$ at which $A_\xi(-r)$ takes nonzero values of alternating signs. Choose $\ell\ge0$ with $e^{\lambda_{\ell-1}t}<r<e^{\lambda_\ell t}$; there is no lower bound if $\ell=0$. For $0<x\le1$ we expand each term of \eqref{eq:Ax3} at $z=-r$ into a geometric series, in powers of $r\,e^{-\lambda_jt}$ for $j\ge\ell$ and in powers of $r^{-1}e^{\lambda_jt}$ for $j<\ell$. This gives
\[
A_x(-r)=\sum_{m\in\ZZ}(-r)^m\,G_\ell\bigl(t(m+x)\bigr).
\]
Let $\sigma=-(\log r)/t$, which satisfies the hypothesis of Lemma~\ref{lem:PF}, and put $\Lambda(u)=(-1)^{\ell}r^{u}G_\ell(tu)$. By Lemma \ref{lem:PF} and (P1), $\Lambda$ is a continuous P\'olya frequency function with $\Lambda(u)\le Ce^{-\epsilon|u|}$. The function
\[
g(u)=\sum_{m\in\ZZ}(-1)^m\Lambda(u+m)
\]
satisfies $g(u+1)=-g(u)$, and $g(x)=(-1)^{\ell}r^{x}A_x(-r)$ for $0<x\le1$. Suppose that $g$ takes nonzero values of alternating signs at $\xi_0<\xi_1<\xi_2$ in $(0,1]$. Since $g(\xi_0+1)=-g(\xi_0)$, the signs of $g$ alternate along the whole sequence
\[
\xi_0<\xi_1<\xi_2<\xi_0+1<\xi_1+1<\xi_2+1<\xi_0+2<\cdots .
\]
Let $\epsilon_0=\min_i|g(\xi_i)|$ and choose $L$ with $\sum_{|n|\ge L}\sup_{0\le u\le1}\Lambda(u+n)<\epsilon_0$. For an integer $M\ge1$ let $g_M(u)=\sum_{-M-L\le m\le L}(-1)^m\Lambda(u+m)$. Then $|g-g_M|<\epsilon_0$ on $(0,M]$, so $g_M$ has the same signs as $g$ at the $3M$ points of the sequence in $(0,M]$. Therefore $g_M$ has at least $3M-1$ sign changes. But $g_M$ is a combination of $M+2L+1$ translates of $\Lambda$ with alternating coefficients, so by (P2) it has at most $M+2L$ sign changes. For large $M$ this is a contradiction.

\emph{Step 4: monotonicity.} The series \eqref{eq:Ax3} extends $A_x(z)$ to all $x>0$ and is jointly real analytic for $x>0$, $z<0$. Near any fixed $x_0\in(0,1]$, each of the first $i$ simple zeros therefore has an analytic continuation. Small disjoint discs around these zeros contain one zero each by Rouch\'e's theorem. On the compact real intervals between the discs and up to zero, $A_{x_0}$ is bounded away from zero, and local uniform convergence excludes additional zeros there. Thus the continuations preserve their indices. From \eqref{eq:Ax3} and \eqref{eq:jacobi},
\[
A_x(z)-G(tx)=z\sum_{j\ge0}\frac{\rho_j\,e^{-\lambda_jt(1+x)}}{1-z\,e^{-\lambda_jt}}\ \longrightarrow\ z\,A_1(z)\qquad(x\to0^+),
\]
locally uniformly away from the poles, and $G(tx)\to0$. The zeros of $zA_1(z)$ are $0$ and the numbers $\zeta_i(1)$, and they are simple. The same disc and compact-interval argument gives $\zeta_1(x)\to0$ and $\zeta_{i+1}(x)\to\zeta_i(1)$ as $x\to0^+$.

Fix $r_0>0$. By \eqref{eq:Ax3},
$e^{\lambda_0tx}A_x(-r_0)\to\rho_0/(1+r_0e^{-\lambda_0t})\ne0$ as $x\to\infty$.
Consequently $x\mapsto A_x(-r_0)$ has isolated zeros, so no branch $\zeta_i$ is constant on an interval. If $\zeta_i$ took a value twice, its real analyticity would give a strict local maximum or minimum between those points. Near that extremum $x_*$, simplicity gives
$A_x(z)=(z-\zeta_i(x))B(x,z)$ with $B$ of constant nonzero sign in a real neighborhood of $(x_*,\zeta_i(x_*))$. Take nearby points on either side of $x_*$. Choose a negative level $-r$ strictly between $\zeta_i(x_*)$ and both endpoint values, avoiding the countable set $r=e^{\lambda_jt}$. Then $A_x(-r)$ has alternating signs at the three points, contrary to Step~3.

Thus every $\zeta_i$ is strictly monotone. Since $\zeta_1(0^+)=0>\zeta_1(1)$ and $\zeta_{i+1}(0^+)=\zeta_i(1)>\zeta_{i+1}(1)$, all of them are decreasing. This proves (ii). For $0<x<1$ it follows that $\zeta_i(1)<\zeta_i(x)<\zeta_i(0^+)=\zeta_{i-1}(1)$, where $\zeta_0(1)=0$, and this is (iii).
\end{proof}

\begin{remark}\label{rem:zak}
For $r=1$, $A_x(-1)$ is the Zak transform of $u\mapsto G(tu)$ at frequency $1/2$. More generally, if $r\ne e^{\lambda_jt}$ for every $j$, the function $\Lambda$ in Step~3 has a Fourier transform of infinite product type with no Gaussian factor by Lemma~\ref{lem:PF}. Kloos's uniqueness theorem \cite{Kloos} applied to its Zak transform therefore also gives a unique $x\in(0,1]$ with $A_x(-r)=0$. Theorem~\ref{thm:zeros} gives the indexed zero branches and their strict order, including at the exceptional radii. Indeed their ranges $[\zeta_i(1),\zeta_{i-1}(1))$, with $\zeta_0(1)=0$, partition the negative axis.
\end{remark}

\section{Partial fractions and root-of-unity averages for \texorpdfstring{$k=3$}{k=3}}\label{sec:cubic}

We keep $k=3$ and return to $J=A_{1/a}/A_1$. Let
\[
u_i=-\zeta_i(1)>0,\qquad 0<u_1<u_2<\cdots,
\]
so that the numbers $-u_i$ are the zeros of $D$. They depend on $q$. Recall that $T=J(0)=G(t/a)/G(t)$ and $\beta=1-1/a$.

\begin{theorem}\label{thm:pick}
Let $k=3$, $a\ge2$ and $0<q<1$.
\begin{enumerate}
\item The zeros of $J(\cdot,q)$ are the numbers $\zeta_i(1/a)$, and its poles are the numbers $-u_i$. They are negative and simple, and
\[
0>\zeta_1(1/a)>-u_1>\zeta_2(1/a)>-u_2>\cdots .
\]
\item $J$ is holomorphic on $\CC\setminus\{-u_i:i\ge1\}$ and positive on $[0,\infty)$, and
\[
0<\arg J(z)<\arg z\qquad\text{for } 0<\arg z<\pi .
\]
\item There are numbers $w_i=w_i(q)>0$ with $\sum_iw_i/u_i<\infty$ such that
\begin{equation}\label{eq:pf}
J(z)=T+\sum_{i\ge1}w_i\,\frac{z}{z+u_i} .
\end{equation}
\item On $(0,\infty)$ the function $J(x)$ increases and $J(x)/x$ decreases. In particular, if $x_-\le x\le x_+$ and $x_\pm$ are nodes, then
\[
\max\bigl(x_-^{\beta},\,x\,x_+^{\beta-1}\bigr)\le J(x)\le\min\bigl(x_+^{\beta},\,x\,x_-^{\beta-1}\bigr).
\]
\end{enumerate}
\end{theorem}

\begin{proof}
(i) This is Theorem \ref{thm:zeros} with $x=1/a$. The poles $e^{\lambda_jt}$ of $A_{1/a}$ and of $A_1$ cancel in the quotient, and the product formula of Theorem \ref{thm:zeros}(i) gives
\[
J(z)=T\prod_{i\ge1}\frac{1-z/\zeta_i}{1+z/u_i},\qquad \zeta_i=\zeta_i(1/a).
\]
(ii) The first two statements follow from (i) and the product. Let $\im z>0$, and for real $u$ let $\phi(u)=\arg(z-u)\in(0,\pi)$. This function increases with $u$, and $\phi(0)=\arg z$. The $i$-th factor of the product has the argument $\phi(\zeta_i)-\phi(-u_i)$, which is positive because $\zeta_i>-u_i$. Since $-u_i>\zeta_{i+1}$, the sum of the first $n$ of these numbers is at most $\phi(\zeta_1)-\phi(-u_n)<\phi(\zeta_1)$. So the series $\sum_i[\phi(\zeta_i)-\phi(-u_i)]$ converges to a value of $\arg J(z)$ which is positive and at most $\phi(\zeta_1)<\arg z$.

(iii) By (ii), $J$ is analytic in the upper half plane and has a positive imaginary part there. Every such function has a representation
\[
J(z)=b_0+b_1z+\int_{\RR}\Bigl[\frac1{\lambda-z}-\frac{\lambda}{1+\lambda^2}\Bigr]d\mu(\lambda),\qquad b_1\ge0,\quad\int_{\RR}\frac{d\mu(\lambda)}{1+\lambda^2}<\infty,
\]
with a real number $b_0$ and a positive measure $\mu$ \cite[Chapter II]{Donoghue}. Since $J$ is meromorphic in $\CC$ and real on the real axis, $\mu$ consists of point masses $m_i>0$ at the poles $-u_i$; for meromorphic functions this is Chebotarev's theorem \cite[Chapter VII]{Levin}. Subtracting the value at $z=0$ we get
\[
J(z)=T+b_1z+\sum_i\frac{m_i}{u_i}\,\frac{z}{z+u_i},\qquad \sum_i\frac{m_i}{u_i^2}<\infty .
\]
Put $w_i=m_i/u_i$. It remains to show that $b_1=0$. For $x>0$ the formula gives $J(x)\ge b_1x$. At the nodes $x_j=e^{\lambda_jt}$ we have $J(x_j)=x_j^{\beta}$ by \eqref{eq:interp}, so that $b_1\le x_j^{\beta-1}\to0$.

(iv) By \eqref{eq:pf}, $J(x)$ increases and $J(x)/x=T/x+\sum_iw_i/(x+u_i)$ decreases. Combine this with $J(x_\pm)=x_\pm^{\beta}$.
\end{proof}

In other words, $J(z)/z=T/z+\sum_iw_i/(z+u_i)$ is the Stieltjes transform of the positive measure $T\delta_0+\sum_iw_i\delta_{u_i}$ on $[0,\infty)$. The power $z^{\beta-1}$ is the Stieltjes transform of $\frac{\sin\pi\beta}{\pi}u^{\beta-1}du$, and $J(z)/z$ agrees with it at the nodes.

\begin{theorem}\label{thm:sections}
Let $k=3$ and $a\ge2$.
\begin{enumerate}
\item For $d\ge1$ and $q\in(0,1)$ assume, when $d$ is even, that $u_i(q)\ne1$ for every $i$. Then
\begin{equation}\label{eq:Wpoles}
W_d=T+\sum_{i\ge1}\frac{w_i}{1-(-u_i)^d} .
\end{equation}
\item If $d$ is odd, then $W_d$ is finite and real analytic on $0<q<1$, and $W_d(q)>T(q)$.
\item If $d$ is even, then the poles of $W_d$ in $(0,1)$ are exactly the zeros of the function $q\mapsto D(-1,q)$, that is, the points where some $u_i(q)$ equals $1$. At a simple zero $q_*$, with $u_i(q_*)=1$, the pole is simple and
\[
\Res_{q=q_*}W_d=-\frac{w_i(q_*)}{d\;u_i'(q_*)} .
\]
So the residue is positive exactly when the pole $-u_i$ of $J$ moves towards the origin as $q$ increases.
\item Let $q_{\mathrm c}$ be the smallest zero of $D(-1,q)$ in $(0,1)$. For $0<q<q_{\mathrm c}$ all $u_i$ are larger than $1$, and
\[
W_2<W_4<W_6<\cdots<T<\cdots<W_5<W_3<W_1 .
\]
Moreover $W_d-T\sim(-1)^{d-1}w_1u_1^{-d}$ as $d\to\infty$.
\item For every $q$ such that no $u_i(q)$ equals $1$,
\[
L(q):=\lim_{d\to\infty}W_d(q)=T(q)+\sum_{u_i<1}w_i ,
\]
and the convergence is geometric. Thus $L=T$ exactly when $J$ has no pole in the unit disc, in particular for $q<q_{\mathrm c}$.
\end{enumerate}
\end{theorem}

\begin{proof}
(i) Let $u>0$ and $\omega^d=1$. We expand $\omega/(\omega+u)$ in powers of $u/\omega$ if $u<1$ and in powers of $\omega/u$ if $u>1$, and we average over $\omega$. In both cases $\frac1d\sum_{\omega^d=1}\omega/(\omega+u)=1/(1-(-u)^d)$, and for $u=1$ and odd $d$ the value is $\frac12$ by continuity. Now average \eqref{eq:pf}, whose series converges absolutely at each of the finitely many evaluation points since $\sum_i w_i/u_i<\infty$.

(ii) By Theorem \ref{thm:zeros} the zeros of $D(\cdot,q)$ are negative. So $D(\omega,q)\ne0$ for every root of unity $\omega\ne-1$ and every $q$, and $J(\omega,q)$ is real analytic in $q$. For odd $d$ each term of the sum in \eqref{eq:Wpoles} is $w_i/(1+u_i^d)>0$.

(iii) For even $d$ only $J(-1,q)=N(-1,q)/D(-1,q)$ can be infinite, and $N(-1,q)\ne0$ at a zero of $D(-1,q)$ by the strict interlacing in Theorem \ref{thm:pick}(i). So the poles of $W_d$ are the zeros of $q\mapsto D(-1,q)$. Let $q_*$ be a simple zero. Then $W_d-\frac1dJ(-1,q)$ is analytic near $q_*$, and the residue of $J(-1,q)$ at $q_*$ is $N/\partial_qD$, evaluated at $(z,q)=(-1,q_*)$. The zero $z=-1$ of $D(\cdot,q_*)$ is simple, so $\partial_zD\ne0$ there. By the implicit function theorem the zero $-u_i(q)$ of $D(\cdot,q)$ near $-1$ is an analytic function of $q$, and $u_i'(q_*)=\partial_qD/\partial_zD\ne0$. Comparing the residues of $J$ at $z=-u_i$ in \eqref{eq:pf} and in $J=N/D$ gives $w_iu_i=-N/\partial_zD$ at $z=-u_i$. For $q=q_*$ we have $u_i=1$, and dividing the two relations gives $N/\partial_qD=-w_i(q_*)/u_i'(q_*)$.

(iv) The function $q\mapsto D(-1,q)$ is real analytic on $(0,1)$. It has zeros by Theorem \ref{thm:poles}, and it has no zeros near $q=0$ by Proposition \ref{prop:EK}. So $q_{\mathrm c}$ exists. The number of zeros of $D(\cdot,q)$ in the open unit disc does not change with $q$ as long as $D(\cdot,q)$ has no zero on the unit circle. By Theorem \ref{thm:pick} such a zero can only be $z=-1$. For small $q$ there is no zero in the closed unit disc, by Proposition \ref{prop:EK}. Hence $u_i>1$ for all $i$ when $q<q_{\mathrm c}$. Then \eqref{eq:Wpoles} gives $W_d-T=\sum_iw_i/(1+u_i^d)$ for odd $d$, which is positive and decreases with $d$, and $W_d-T=-\sum_iw_i/(u_i^d-1)$ for even $d$, which is negative and increases with $d$. The $i=1$ term is asymptotic to $(-1)^{d-1}w_1u_1^{-d}$. The sum over $i\ge2$ is $O(u_2^{1-d}\sum_{i\ge2}w_i/u_i)$ by the estimate below.

(v) Put $r=\max(\{u_i:u_i<1\}\cup\{0\})$ and $U=\min\{u_i:u_i>1\}$. The poles form a discrete sequence tending to infinity, so $r<1<U$. Formula \eqref{eq:Wpoles} gives
\[
\left|W_d-T-\sum_{u_i<1}w_i\right|
\le \frac{r^d}{1-r}\sum_{u_i<1}w_i
+\frac{U^{1-d}}{1-U^{-1}}\sum_{u_i>1}\frac{w_i}{u_i}.
\]
The first sum is finite and the second converges by Theorem~\ref{thm:pick}(iii).
\end{proof}

\begin{theorem}\label{thm:limit}
Let $k=3$ and $a\ge2$.
\begin{enumerate}
\item Let $K$ be a compact subset of $\CC\setminus(-\infty,0]$. There are constants $C$ and $t_0>0$ such that
\[
\bigl|J\bigl(z,e^{-t/a}\bigr)-z^{\beta}\bigr|\le C\,|z|^{\beta}\,e^{-c(z)/\sqrt t}\qquad(z\in K,\ 0<t\le t_0),
\]
where $c(z)$ is given by \eqref{eq:cz} with $\kappa=\pi\sqrt2$.
\item For every continuous function $\varphi$ with compact support in $(0,\infty)$,
\[
\sum_{i\ge1}w_i(q)\,\varphi\bigl(u_i(q)\bigr)\ \longrightarrow\ \frac{\sin(\pi/a)}{\pi}\int_0^\infty\varphi(u)\,u^{-1/a}\,du\qquad(q\to1^-).
\]
\item With $L(q)$ as in Theorem \ref{thm:sections}(v),
\[
\lim_{q\to1^-}L(q)=\frac{a\,\sin(\pi/a)}{\pi\,(a-1)},
\]
where $q$ avoids the zeros of $D(-1,q)$.
\end{enumerate}
\end{theorem}

\begin{proof}
(i) Take the principal logarithm \(w=\log z\), and put
\[
 s_n=\frac{2\pi i n-w}{t},\qquad r_n=\re\sqrt{2\pi i n-w},\qquad
 v_n=\frac{\pi}{2}\sqrt{7+8s_n},\qquad \mu_n=\re v_n .
\]
By \eqref{eq:L3}, \(\LL(s_n)=2\pi\sech v_n\). For \(\mu>0\),
the identity \(|\cosh(\mu+i\nu)|^2=\sinh^2\mu+\cos^2\nu\) gives
\begin{equation}\label{eq:sechbound}
 \frac1{\cosh\mu}\le|\sech (\mu+i\nu)|
 \le\frac1{\sinh\mu}.
\end{equation}
For principal square roots, \(|\sqrt{X+7}-\sqrt X|\le\sqrt7\).
Indeed, its product with \(|\sqrt{X+7}+\sqrt X|\) equals \(7\), and
the two roots lie in the same closed quadrant. Taking \(X=8s_n\) gives
\(\mu_n=\kappa r_n/\sqrt t+O(1)\) with an absolute error.

Write \(r_*=\min(r_{-1},r_1)\). For fixed real \(A\), the quantity
\(\re\sqrt{A+iB}=\sqrt{(\sqrt{A^2+B^2}+A)/2}\) increases strictly with
\(|B|\). Since \(|\im w|<\pi\), we have \(r_n\ge r_*>r_0\ge0\) for
every \(n\ne0\). Both \(r_*\) and \(r_*-r_0\) have positive minima
on \(K\) by continuity. Also \(r_n/\sqrt{\pi|n|}\to1\) uniformly on
\(K\) as \(|n|\to\infty\), since \(\log K\) is bounded.
Choose \(0<t_0\le1\), depending on \(K\), so that \(\mu_n\ge1\) for
all \(z\in K,\ n\ne0,\ 0<t\le t_0\).

For \(n\ne0\) set \(H_n=\cosh v_0/\cosh v_n\). Where \(\LL(s_0)\) is finite, \(H_n=\LL(s_n)/\LL(s_0)\).
At a pole of \(\LL(s_0)\), the numerator \(\cosh v_0\) vanishes, so \(H_n=0\).
By \eqref{eq:sechbound} and \(|\cosh v_0|\le\cosh\mu_0\),
\[
 |H_n|\le\frac{\cosh\mu_0}{\sinh\mu_n}
     \le C\exp\!\left(-\frac{\kappa(r_n-r_0)}{\sqrt t}\right).
\]
In view of \(c(z)=\kappa(r_*-r_0)\) and \(t\le1\),
\[
 \sum_{n\ne0}|H_n|
 \le C e^{-c(z)/\sqrt t}\sum_{n\ne0}e^{-\kappa(r_n-r_*)}
 \le C_K e^{-c(z)/\sqrt t}.
\]
The last series is uniformly bounded by the uniform growth of \(r_n\)
and the boundedness of \(r_*\) on \(K\). Put
\(R_x=\sum_{n\ne0}e^{2\pi i n x}H_n\). Since \(\min_K c>0\), decreasing
\(t_0\) gives \(|R_1|\le\tfrac12\) throughout \(K\). Formula \eqref{eq:dualJ}
now yields \(J(z,e^{-t/a})/z^\beta=(1+R_{1/a})/(1+R_1)\) away from
the nodes. At a node every \(H_n\) is zero and \eqref{eq:interp} gives the same
identity. Consequently
\(\bigl|J(z,e^{-t/a})/z^\beta-1\bigr|\le 4C_K e^{-c(z)/\sqrt t}\)
throughout \(K\), as claimed.

(ii) and (iii). Let $\mathsf m_q=T\delta_0+\sum_i\frac{w_i}{1+u_i}\delta_{u_i}$. This is a positive measure on $[0,\infty)$ of total mass $J(1,q)$, and by \eqref{eq:pf}
\[
J(x,q)=\int g_x\,d\mathsf m_q\quad(x>0),\qquad g_x(u)=\frac{x(1+u)}{x+u} .
\]
Also $x^{\beta}=\int g_x\,d\mathsf m$ with $d\mathsf m=\frac{\sin\pi\beta}{\pi}\frac{u^{\beta-1}}{1+u}\,du$, because $\int_0^\infty u^{\beta-1}(x+u)^{-1}du=\pi x^{\beta-1}/\sin\pi\beta$. The functions $g_x$ are continuous on the compact space $[0,\infty]$. Their linear span contains the constants, since $g_1=1$, and the functions $1/(x+u)$ for $x\ne1$. These resolvents separate points, including infinity. Writing $h_x(u)=1/(x+u)$, one has $h_xh_y=(h_x-h_y)/(y-x)$ for $x\ne y$. Uniform convergence as $y\to x>0$ includes $h_1$ and the squares $h_x^2$ in the closed span, so it is an algebra. By the Stone--Weierstrass theorem the span is dense in $C([0,\infty])$. By (i), $\int g_x\,d\mathsf m_q\to\int g_x\,d\mathsf m$ for each $x>0$, and the total masses $J(1,q)$ are bounded. Hence $\int g\,d\mathsf m_q\to\int g\,d\mathsf m$ for every $g\in C([0,\infty])$. Taking $g(u)=(1+u)\varphi(u)$ gives (ii). For (iii) note that $L(q)=\int(1+u)\mathbf 1_{[0,1)}(u)\,d\mathsf m_q(u)$ by Theorem \ref{thm:sections}(v). The integrand is bounded and continuous except at $u=1$, and $\mathsf m$ has no mass at this point. So $L(q)\to\int_0^1(1+u)\,d\mathsf m(u)=\frac{\sin\pi\beta}{\pi}\int_0^1u^{\beta-1}du=\sin(\pi\beta)/(\pi\beta)$.
\end{proof}

\begin{corollary}\label{cor:polemotion}
Let $k=3$. The radii $u_i(q)$ of the poles satisfy
$u_i'(q)\le-a u_i(q)/q<0$ for $0<q<1$. Every pole of an even-order average
in this interval is simple and has positive residue. The number of poles
of $J(\cdot,q)$ in $|z|<1$ equals the number of zeros of $D(-1,\cdot)$
in $(0,q)$.
\end{corollary}

\begin{proof}
Write $U_i(t)=u_i(e^{-t/a})$, and let $\rho_{j,M}$ be the residue at
$-\lambda_j$ of $\LL(0)\prod_{h=0}^M(1+s/\lambda_h)^{-1}$.
By \eqref{eq:Lprod},
$\rho_{j,M}=\rho_j\prod_{h>M}(1-\lambda_j/\lambda_h)$ for $j\le M$.
These residues converge to $\rho_j$ and have absolute value at most
$|\rho_j|=2j+1$. Hence, for each $t>0$,
\[
 D_M(z,t):=\sum_{j=0}^M\frac{\rho_{j,M}}{e^{t\lambda_j}-z}
 \longrightarrow D(z,e^{-t/a})
\]
locally uniformly away from the positive poles, by \eqref{eq:Ax3}.

Put $b_j=t\lambda_j$. The finite residue formula shows that $D_M$ is a
nonzero constant times the divided difference
$[b_0,\ldots,b_M]_v\,(e^v-z)^{-1}$.
For real $b_j$ this is the cardinal $L$-spline formula of
\cite[eq.~(4.4), Theorem~3 and the following remark]{MicchelliSharma};
see also \cite{Micchelli}. In their notation the function in the divided difference for
shift $x$ has numerator $e^{xv}$, so our formula uses $x=0$.
For $0<x<1$ its ordered negative zeros decrease in each $b_j$.
At $x\downarrow0$ the leftmost zero escapes to $-\infty$ and the others
tend to the ordered finite zeros at $x=0$, giving the same non-strict
inequality there. Adding $h$ to all $b_j$ multiplies every zero by $e^h$,
as follows by shifting $v$ in the divided difference.

Let $U_{i,M}(t)$ denote the modulus of the $i$th zero of $D_M$ counted
from zero. For $t_2>t_1$ put $h=(t_2-t_1)\lambda_0$. Since
$t_2\lambda_j\ge t_1\lambda_j+h$, the two preceding comparisons give
$U_{i,M}(t_2)\ge e^h U_{i,M}(t_1)$ whenever this zero exists.
By local uniform convergence and Theorem~\ref{thm:zeros}, the ordered
finite zeros converge to the corresponding simple zeros of $D$:
on a compact interval from the origin containing the first $i$ zeros,
Rouch\'e's theorem isolates one zero near each and uniform convergence
excludes any in the intervening gaps. Thus
$U_i(t_2)\ge e^h U_i(t_1)$. Simplicity and the implicit function theorem
give differentiability, whence $U_i'(t)\ge\lambda_0 U_i(t)$.
Now $\lambda_0=1$ and $t=-a\log q$ give the claimed inequality.
Each crossing of $u_i=1$ is transverse, so Theorem~\ref{thm:sections}(iii)
gives the residue sign. All $u_i$ start above one for small $q$
by Proposition~\ref{prop:EK}, and each crosses at most once.
Distinctness of the zeros of $D$ excludes simultaneous crossings, proving the count.
\end{proof}

\section{Asymptotic locations of the poles for \texorpdfstring{$k=3$}{k=3}}\label{sec:where}

\begin{proposition}\label{prop:poleeq}
Let $k=3$ and $t>0$. For real $\sigma$,
\begin{equation}\label{eq:poleeq}
D\bigl(-e^{-\sigma}\bigr)=-\frac{4\pi\,e^{\sigma}}{t}\sum_{\substack{n\ge1\\ n\ \mathrm{odd}}}\re\,\sech\Bigl(\frac\pi2\sqrt{7+\frac{8(\sigma+i\pi n)}{t}}\Bigr).
\end{equation}
So the poles of $J(\cdot,q)$ are the points $z=-e^{-\sigma}$ for which the sum in \eqref{eq:poleeq} vanishes.
\end{proposition}

\begin{proof}
In \eqref{eq:dualA} take $x=1$ and $w=i\pi-\sigma$. Then $e^{2\pi i m-w}=-e^\sigma$ and the arguments of $\LL$ are $(\sigma+i\pi(2m-1))/t$. Pairing the terms $m$ and $1-m$ and inserting \eqref{eq:L3} gives \eqref{eq:poleeq}. Strict interlacing gives the assertion about poles.
\end{proof}

\begin{remark}\label{rem:ideal}
Write $\frac\pi2\sqrt{7+8(\sigma+i\pi)/t}=\mu+i\nu$. Since
\[
\re\sech(\mu+i\nu)=\frac{\cosh\mu\,\cos\nu}{\sinh^2\mu+\cos^2\nu},
\]
the first term of the sum in \eqref{eq:poleeq} vanishes exactly when $\nu=\pi(j+\frac12)$ for an integer $j\ge0$, that is, when
\[
\sigma=\frac{2\pi^2}{(2j+1)^2\,t}-\lambda_j\,t .
\]
At $\sigma=0$ this gives the numbers $t_j^*$ below.
\end{remark}

The next theorem gives the displacement of the zeros of $D(-1)$ from $t_j^*$.

\begin{theorem}\label{thm:tj}
Let $k=3$ and let $t_j$ be the zeros of $D(-1)$ as in Theorem \ref{thm:poles}. Put $h=j+\frac12$ and
\[
t_j^*=\frac{\pi}{h\sqrt{h^2+7/4}}=\frac{4\pi}{(2j+1)\sqrt{(2j+1)^2+7}} .
\]
Then, as $j\to\infty$,
\[
t_j-t_j^*=\frac{2(-1)^{j+1}}{h^3}\,e^{-(\sqrt3-1)\pi h}\Bigl[\cos\bigl(\sqrt3\,\pi h\bigr)+O\Bigl(\frac1h\Bigr)\Bigr].
\]
The numbers $(-1)^{j+1}\cos(\sqrt3\pi h)$ are dense in $[-1,1]$, so the exponent $(\sqrt3-1)\pi$ cannot be improved.
\end{theorem}

\begin{proof}
For odd $n\ge1$ write $\frac\pi2\sqrt{7+8\pi in/t}=\mu_n+i\nu_n$, and let $H_n=\re\sech(\mu_n+i\nu_n)=\cosh\mu_n\cos\nu_n/(\sinh^2\mu_n+\cos^2\nu_n)$. By \eqref{eq:poleeq} with $\sigma=0$, the $t_j$ are the zeros of $\sum_{n\ \mathrm{odd}}H_n(t)$. Squaring $\mu_1+i\nu_1$ gives
\[
\mu_1^2-\nu_1^2=\frac{7\pi^2}{4},\qquad \mu_1\nu_1=\frac{\pi^3}{t} .
\]
So $\nu_1$ is a decreasing function of $t$, and $\nu_1=\pi h$ exactly for $t=t_j^*$. Moreover $\nu_1=\Theta(t)+O(1/\Theta(t))$ with $\Theta(t)=\pi^{3/2}/\sqrt t$, and $\psi=\Theta$ for $k=3$. So for large $j$ the interval $|\Theta(t)-\pi h|<\frac\pi4$ of Theorem \ref{thm:poles} contains both $t_j$ and $t_j^*$. Throughout this interval $|\nu_1-\pi h|<\frac\pi3$ and $\mu_1=\pi h+O(1)$.

Since $\re\sqrt{A+iB}\ge\sqrt{B/2}$ for $A,B\ge0$, we have $\mu_n\ge\pi^{3/2}\sqrt{n/t}\ge\sqrt n\,(\pi h-C)$ on this interval for a fixed $C$.
By \eqref{eq:sechbound}, $|H_n|\le1/\sinh\mu_n$. For $\pi h-C\ge1$,
\[
 \sum_{n\ge5}|H_n|\ll e^{-\sqrt5(\pi h-C)}
                 \sum_{n\ge5}e^{-(\sqrt n-\sqrt5)}
             \ll e^{-\sqrt5\pi h},
\]
where $n$ runs through the odd integers. Also $H_3=2e^{-\mu_3}\cos\nu_3\,(1+O(e^{-2\mu_3}))$. Write $\nu_1(t_j)=\pi h+\theta_j$. Then $\cos\nu_1=(-1)^{j+1}\sin\theta_j$, and the equation $H_1=-H_3-H_5-\cdots$ at $t=t_j$ reads
\[
(-1)^{j+1}\sin\theta_j\cdot 2e^{-\mu_1}\bigl(1+O(e^{-2\mu_1})\bigr)=-2e^{-\mu_3}\cos\nu_3\,\bigl(1+O(e^{-2\mu_3})\bigr)+O\bigl(e^{-\sqrt5\pi h}\bigr).
\]
First this shows that $\theta_j=O(e^{-(\sqrt3-1)\pi h})$. Then $t_j=\pi h^{-2}(1+O(h^{-2}))$, so that $\mu_1=\pi h+O(1/h)$ and $\mu_3+i\nu_3=\frac\pi2\sqrt{7+24ih^2(1+O(h^{-2}))}=\sqrt3\,\pi h\,(1+i)+O(1/h)$ at $t=t_j$. Therefore
\[
\theta_j=(-1)^{j}e^{-(\sqrt3-1)\pi h}\bigl[\cos(\sqrt3\pi h)+O(1/h)\bigr].
\]
Finally $\nu_1(t_j)-\nu_1(t_j^*)=\theta_j$, and differentiating $(\mu_1+i\nu_1)^2$ gives $\nu_1'(t)=-\pi^3\mu_1/(t^2(\mu_1^2+\nu_1^2))=-\frac12h^3(1+O(h^{-2}))$ between $t_j$ and $t_j^*$. So $t_j-t_j^*=-2\theta_jh^{-3}(1+O(h^{-2}))$. Already along $j=2m$, the angles $2\pi\sqrt3\,m+\sqrt3\,\pi/2$ are dense modulo $2\pi$. This gives the asserted density and a subsequence on which the absolute value of the bracket stays bounded away from zero.
\end{proof}

\begin{remark}\label{rem:bracket}
The term $O(1/h)$ is explicit. At $t=t_j^*$ one has $\mu_1=\pi h+\frac{7\pi}{8h}+O(h^{-3})$ and $\mu_3+i\nu_3=\sqrt3\,\pi h\,(1+i)+\frac{7\sqrt3\,\pi}{24h}(2+i)+O(h^{-3})$, and the proof shows that the bracket in Theorem \ref{thm:tj} equals
\[
\Bigl(1-\frac{c'}{h}\Bigr)\cos\Bigl(\sqrt3\,\pi h+\frac{c''}{h}\Bigr)+O\Bigl(\frac1{h^2}\Bigr),\qquad c'=\frac{7\pi(2\sqrt3-3)}{24},\quad c''=\frac{7\sqrt3\,\pi}{24}.
\]
\end{remark}

\section{Further remarks and open questions}\label{sec:questions}

\begin{remark}\label{rem:others}
The interpolation property \eqref{eq:interp} holds at the nodes with $f_{n-1}\ne0$, with the same proof, when $\Phi^k$ is replaced by any nonzero power series $F=\sum_nf_nx^n$ which converges in the unit disc. Beyond this, the proofs of Lemma \ref{lem:PF} and of Theorems \ref{thm:zeros} and \ref{thm:pick} use only the following properties of the function $G(v)=e^{-v}F(e^{-v})$ with $F=\Phi^3$. The function $G$ is positive on $(0,\infty)$ and $G(0^+)=0$. It has an expansion $G(v)=\sum_j\rho_je^{-\lambda_jv}$ with nonzero integers $\rho_j$ and positive integers $\lambda_0<\lambda_1<\cdots$ such that $\sum_j|\rho_j|e^{-\lambda_j\varepsilon}<\infty$ for every $\varepsilon>0$ and $\sum_{i<j}(\lambda_j-\lambda_i)/\lambda_j\to\infty$. Its Laplace transform is $\LL(0)\prod_j\lambda_j/(\lambda_j+s)$ with $\LL(0)>0$. All this holds also for the theta functions
\begin{gather*}
\sum_{n\in\ZZ}(-1)^nx^{n^2}=\frac{\Phi(x)^2}{\Phi(x^2)},\qquad
\sum_{n\in\ZZ}(6n+1)x^{n(3n+1)/2}=\frac{\Phi(x)^5}{\Phi(x^2)^2},\\
\sum_{n\in\ZZ}(3n+1)x^{3n^2+2n}=\frac{\Phi(x)^2\Phi(x^4)^2}{\Phi(x^2)} .
\end{gather*}
The first identity is due to Gauss, and the other two follow from the quintuple product identity \cite{Cooper}, \cite[Chapter 16]{Berndt}. The three Laplace transforms are, respectively,
\[
\frac{\pi}{\sqrt{s+1}\,\sinh\bigl(\pi\sqrt{s+1}\bigr)},\qquad
\frac{4\sqrt3\,\pi}{2\cosh\bigl(\frac\pi3\sqrt{24s+23}\bigr)-1},\qquad
\frac{\sqrt3\,\pi}{2\cosh\bigl(\frac{2\pi}{3}\sqrt{3s+2}\bigr)+1} .
\]
For the first transform, integrate termwise and use the partial fractions of $\csc$ at $i\pi\sqrt{s+1}$. For the other two, pair the terms $n$ and $-n$ before integrating. For $n\ge1$ the pair is bounded by $C(1+n^2v)e^{-cn^2v}$, by the mean value theorem applied to $y e^{-bvy^2}$ with $y$ between $dn-1$ and $dn+1$, where $(d,b)=(6,1/24)$ or $(3,1/3)$. Its $L^1$ norm is $O(n^{-2})$. Completing the square then reduces the sums to the real part of the cotangent expansion at $\pi(1+iy)/d$, with $y=\sqrt{24s+23}$ or $\sqrt{3s+2}$, initially for $s\ge0$ \cite[4.22.3, 4.22.5]{DLMF}. In each case the denominator is entire of order $1/2$ with simple zeros at the $-\lambda_j$, so Hadamard factorization gives the required product. The increasing exponents satisfy $\lambda_j\asymp j^2$ and $\lambda_j-\lambda_{\lfloor j/2\rfloor}\gg j^2$ by the displayed quadratic formulas. Positivity and the limit $G(0^+)=0$ follow from the products and the $\eta$ transformation. So for these three functions $F$ the zeros and the poles of the corresponding quotient $J$ are again negative, simple and interlacing, and $J$ has an expansion \eqref{eq:pf} with positive weights.

\end{remark}

\begin{proposition}\label{prop:onlycube}
For a positive integer $k$, the function $G(v)=e^{-v}\Phi(e^{-v})^k$ on $v>0$, extended by zero to $v\le0$, is a P\'olya frequency function if and only if $k=3$.
\end{proposition}

\begin{proof}
Suppose $G$ is a P\'olya frequency function. By Schoenberg's one-sided factorization \cite{Sch51}, $\LL(0)/\LL(s)=e^{\delta s}\prod_j(1+\alpha_js)$, with $\delta\ge0$, $\alpha_j>0$ and $\sum_j\alpha_j<\infty$. Proposition~\ref{prop:laplace}(i) identifies the zeros of this product, with multiplicity one, as the numbers $-(n+1)$ for which $f_n\ne0$. Hence $\sum_{f_n\ne0}1/(n+1)<\infty$. For $k\notin\{1,3\}$, Ono's coefficient estimate \cite[Corollary~2]{OnoLacunarity} gives $B(X):=\#\{n+1\le X:f_n\ne0\}\gg_k X/\log X$. Partial summation would then give
\[
 \sum_{\substack{n+1\le X\\ f_n\ne0}}\frac{1}{n+1}
 =\frac{B(X)}{X}+\int_1^X\frac{B(v)}{v^2}\,dv\ \longrightarrow\ \infty.
\]
For $k=1$, \eqref{eq:L1} gives $\LL(-59/24)=0$ by taking $y=6i$, whereas the reciprocal of an entire function has no zeros. Lemma~\ref{lem:PF} proves the assertion for $k=3$.
\end{proof}

\begin{remark}\label{rem:real}
For fixed real $a>1$ take $q\in(0,1)$ and $q^a=\exp(a\log q)$. The sampling formulas then have $x=1/a\in(0,1)$ and $\beta=1-1/a$. The interpolation, interlacing and real-parameter limit results of Sections~\ref{sec:interp} and~\ref{sec:laplace}--\ref{sec:where} hold as stated with this convention. The assertion of single-valued meromorphy in the complex $q$-disc uses integral $a$.
\end{remark}

Remark~\ref{rem:bary} gives a positive pole of $J$ for $k=2$ and small positive $q$. The location of zeros of $D$ inside the unit disc is a separate question.

\begin{conjecture}\label{conj:allk}
For every $k\ge1$ and every $0<q<1$, all zeros of $D(\cdot,q)$ in the closed unit disc are negative and simple.
\end{conjecture}

Theorem~\ref{thm:zeros} proves the conjecture for $k=3$. In general, the conjecture would make every average of odd order finite on $(0,1)$, since its sampling points exclude $-1$. Before the first zero of $D(-1,\cdot)$, the argument principle and Proposition~\ref{prop:EK} would give a zero-free closed unit disc. Equation~\eqref{eq:Wc} would then imply $W_d\to T$.

For any fixed $q$ with $D(\cdot,q)$ nonzero on the unit circle, $J(\cdot,q)$ is holomorphic on an annulus containing that circle. Its Laurent coefficients on the annulus decay geometrically in both directions. Averaging the Laurent series retains only indices divisible by $d$, so $W_d$ tends geometrically to its constant coefficient. By the residue theorem this coefficient is
\[
 T(q)+\sum_{\substack{p\ {\rm pole\ of}\ J(\cdot,q)\\ |p|<1}}
       \Res_{z=p}\frac{J(z,q)}{z}.
\]
For $k=3$ the residues in this formula are the weights $w_i$ of Theorem~\ref{thm:pick}.

Does this constant coefficient tend to $a\sin(\pi/a)/(\pi(a-1))$ as $q\to1^-$ through these admissible values, for every $k$? Theorem~\ref{thm:limit} gives this limit for $k=3$. For general $k$, Theorem~\ref{thm:radial} gives convergence off $z=-1$ on the circle, but leaves the contribution near $-1$ undetermined.

\section*{Statements and Declarations}
\noindent\textbf{Funding.} This research received no external funding.

\smallskip\noindent\textbf{Competing interests.} The author declares no competing interests relevant to this work.

\smallskip\noindent\textbf{Data availability.} The results are proved analytically and use no external datasets.

\end{document}